\documentclass[10pt,a4paper]{amsart}
\usepackage{amsmath,amsfonts,amsthm,amssymb,graphicx,mathtools}
\usepackage[stretch=10]{microtype}
\usepackage[unicode]{hyperref}
\usepackage{xcolor}
\definecolor{dark-red}{rgb}{0.4,0.15,0.15}
\definecolor{dark-blue}{rgb}{0.15,0.15,0.4}
\definecolor{medium-blue}{rgb}{0,0,0.5}
\hypersetup{
	pdftitle={Density Theorems for Exceptional Eigenvalues for Congruence Subgroups},
	pdfauthor={Peter Humphries},
	pdfnewwindow=true,   
	colorlinks, linkcolor={dark-red},
	citecolor={dark-blue}, urlcolor={medium-blue}
}
\makeatletter
\newcommand*{\defeq}{\mathrel{\rlap{%
                     \raisebox{0.3ex}{$\m@th\cdot$}}%
                     \raisebox{-0.3ex}{$\m@th\cdot$}}%
                     =}
\newcommand*{\eqdef}{\mathrel{=\llap{%
                     \raisebox{0.3ex}{$\m@th\cdot$}}%
                     \llap{\raisebox{-0.3ex}{$\m@th\cdot$}}}%
                     }
\makeatother
\renewcommand\aa{\mathfrak{a}}
\newcommand\A{\mathbb{A}}
\renewcommand\AA{\mathcal{A}}
\newcommand\BB{\mathcal{B}}
\renewcommand\C{\mathbb{C}}
\newcommand\dee{\partial}
\newcommand\EE{\mathcal{E}}
\newcommand\GL{\mathrm{GL}}
\newcommand\Hb{\mathbb{H}}
\newcommand\OO{\mathcal{O}}
\newcommand\PP{\mathcal{P}}
\newcommand\Q{\mathbb{Q}}
\newcommand\qq{\mathfrak{q}}
\newcommand\R{\mathbb{R}}
\newcommand\SL{\mathrm{SL}}
\newcommand\St{\mathrm{St}}
\newcommand\Z{\mathbb{Z}}
\def\e{\varepsilon}
\DeclareMathOperator{\ad}{ad}
\DeclareMathOperator{\arsinh}{arsinh}
\DeclareMathOperator*{\Res}{Res}
\DeclareMathOperator{\sgn}{sgn}
\DeclareMathOperator{\vol}{vol}
\numberwithin{equation}{section}
\newtheorem{theorem}[equation]{Theorem}
\newtheorem{corollary}[equation]{Corollary}
\newtheorem{lemma}[equation]{Lemma}
\newtheorem{proposition}[equation]{Proposition}
\theoremstyle{remark}
\newtheorem{remark}[equation]{Remark}
\begin{document}

\title[Density Theorems for Exceptional Eigenvalues]{Density Theorems for Exceptional Eigenvalues for Congruence Subgroups}

\author{Peter Humphries}


\address{Department of Mathematics, University College London, Gower Street, London WC1E 6BT, United Kingdom}

\email{\href{mailto:pclhumphries@gmail.com}{pclhumphries@gmail.com}}

\keywords{Selberg eigenvalue conjecture, Ramanujan conjecture}

\subjclass[2010]{11F72 (primary); 11F30 (secondary)}

\begin{abstract}
Using the Kuznetsov formula, we prove several density theorems for exceptional Hecke and Laplacian eigenvalues of Maa\ss{} cusp forms of weight $0$ or $1$ for the congruence subgroups $\Gamma_0(q)$, $\Gamma_1(q)$, and $\Gamma(q)$. These improve and extend upon results of Sarnak and Huxley, who prove similar but slightly weaker results via the Selberg trace formula.
\end{abstract}

\maketitle

\section{Introduction}

Let $\kappa \in \{0,1\}$, let $\Gamma$ be a congruence subgroup of $\SL_2(\Z)$, and let $\chi$ be a congruence character of $\Gamma$ satisfying $\chi(-I) = (-1)^{\kappa}$ should $-I$ be a member of $\Gamma$. Denote by $\AA_{\kappa}(\Gamma,\chi)$ the space spanned by Maa\ss{} cusp forms of weight $\kappa$, level $\Gamma$, and nebentypus $\chi$, namely the $L^2$-closure of the space of smooth functions $f : \Hb \to \C$ satisfying
\begin{itemize}
\item $f(\gamma z) = \chi(\gamma) j_{\gamma}(z)^{\kappa} f(z)$ for all $\gamma \in \Gamma$ and $z \in \Hb$, where for $\gamma = \left(\begin{smallmatrix} a & b \\ c & d \end{smallmatrix}\right) \in \Gamma$,
\[j_{\gamma}(z) \defeq \frac{cz + d}{|cz + d|},\]
\item $f$ is an eigenfunction of the weight $\kappa$ Laplacian
\[\Delta_{\kappa} \defeq -y^2 \left(\frac{\dee^2}{\dee x^2} + \frac{\dee^2}{\dee y^2}\right) + i\kappa y \frac{\dee}{\dee x},\]
\item $f$ is of moderate growth, and
\item the constant term is zero in the Fourier expansion of $f$ at every cusp $\aa$ of $\Gamma \backslash \Hb$ that is singular with respect to $\chi$.
\end{itemize}

We may choose a basis $\BB_{\kappa}(\Gamma,\chi)$ of the complex vector space $\AA_{\kappa}(\Gamma,\chi)$ consisting of Hecke eigenforms. For $f \in \BB_{\kappa}(\Gamma,\chi)$, we let $\lambda_f = 1/4 + t_f^2$ denote the eigenvalue of the weight $\kappa$ Laplacian, where either $t_f \in [0,\infty)$ or $it_f \in (0,1/2)$. Similarly, we let $\lambda_f(p)$ denote the eigenvalue of the Hecke operator $T_p$ at a prime $p$, so that $|\lambda_f(p)| < p^{1/2} + p^{-1/2}$. The generalised Ramanujan conjecture states that $t_f$ is real and that $|\lambda_f(p)| \leq 2$ for every prime $p$. Exceptions to this conjecture are called exceptional eigenvalues. It is known that exceptional Laplacian eigenvalues cannot occur if $\kappa = 1$, while for $\kappa = 0$ there are no exceptional Laplacian eigenvalues for Maa\ss{} cusp forms of squarefree conductor less than $857$ \cite[Theorem 1]{BS}. The best current bounds towards the generalised Ramanujan conjecture are due to Kim and Sarnak \cite[Proposition 2 of Appendix 2]{Kim}; they show that
\[\lambda_f \geq \frac{1}{4} - \left(\frac{7}{64}\right)^2, \qquad \left|\lambda_f(p)\right| \leq p^{7/64} + p^{-7/64}.\]

\subsection{Results}

In this paper, we use the Kuznetsov formula to prove density results for exceptional eigenvalues for the congruence subgroups
\begin{align*}
\Gamma_0(q) & \defeq \left\{\begin{pmatrix} a & b \\ c & d \end{pmatrix} \in \SL_2(\Z) : c \equiv 0 \hspace{-.2cm} \pmod{q}\right\},	\\
\Gamma_1(q) & \defeq \left\{\begin{pmatrix} a & b \\ c & d \end{pmatrix} \in \SL_2(\Z) : a,d \equiv 1 \hspace{-.2cm} \pmod{q}, \ c \equiv 0 \hspace{-.2cm} \pmod{q}\right\},	\\
\Gamma(q) & \defeq \left\{\begin{pmatrix} a & b \\ c & d \end{pmatrix} \in \SL_2(\Z) : a,d \equiv 1 \hspace{-.2cm} \pmod{q}, \ b,c \equiv 0 \hspace{-.2cm} \pmod{q}\right\},
\end{align*}
with $\chi$ equal to the trivial character for the latter two congruence subgroups. Recall that
\[\vol(\Gamma \backslash \Hb) = \frac{\pi}{3} \left[\SL_2(\Z) : \Gamma\right] = \begin{dcases*}
\frac{\pi}{3} q \prod_{p \mid q} \left(1 + \frac{1}{p}\right) & if $\Gamma = \Gamma_0(q)$,	\\
\frac{\pi}{3} q^2 \prod_{p \mid q} \left(1 - \frac{1}{p^2}\right) & if $\Gamma = \Gamma_1(q)$,	\\
\frac{\pi}{3} q^3 \prod_{p \mid q} \left(1 - \frac{1}{p^2}\right) & if $\Gamma = \Gamma(q)$.
\end{dcases*}\]
When $\chi$ is the trivial character, we write $\BB_{\kappa}(\Gamma)$ in place of $\BB_{\kappa}(\Gamma,\chi)$, while when $\Gamma = \Gamma_0(q)$, we write this as $\BB_{\kappa}(q,\chi)$. Given positive integers $q$ and $q_{\chi}$ with $q_{\chi} \mid q$, we factorise $q = \prod_{p^{\alpha} \parallel q} p^{\alpha}$ and $q_{\chi} = \prod_{p^{\gamma} \parallel q_{\chi}} p^{\gamma}$, and define
\[\dot{Q} = \dot{Q}(q,q_{\chi}) = \prod_{\substack{p^{\alpha} \parallel q \\ p^{\gamma} \parallel q_{\chi}}} \dot{Q}(p^{\alpha},p^{\gamma}), \qquad \ddot{Q} = \ddot{Q}(q,q_{\chi}) = \prod_{\substack{p^{\alpha} \parallel q \\ p^{\gamma} \parallel q_{\chi}}} \ddot{Q}(p^{\alpha},p^{\gamma})\]
with
\begin{align*}
\dot{Q}(p^{\alpha},p^{\gamma}) & \defeq \begin{dcases*}
p^{\lfloor \frac{3\alpha + 1}{4} \rfloor - \frac{\alpha}{2}} & if $p$ is odd and $\alpha = \gamma \geq 3$,	\\
2^{\lfloor \frac{3\alpha + 1}{4} \rfloor - \frac{\alpha}{2}} & if $p = 2$ and $\gamma + 1 \geq \alpha \geq 3$,	\\
1 & otherwise,
\end{dcases*}	\\
\ddot{Q}(p^{\alpha},p^{\gamma}) & \defeq \begin{dcases*}
p & if $p$ is odd and $\alpha = \gamma \geq 3$,	\\
4 & if $p = 2$ and $\alpha = \gamma \geq 3$,	\\
2 & if $p = 2$ and $\alpha = \gamma + 1 \geq 3$,	\\
1 & otherwise.
\end{dcases*}
\end{align*}

\begin{theorem}\label{Sarnakthm}
For any fixed finite collection of primes $\PP$ not dividing $q$, any $\alpha_p \in (2, p^{1/2} + p^{-1/2})$ and $0 \leq \mu_p \leq 1$ for all $p \in \PP$ with $\sum_{p \in \PP} \mu_p = 1$, we have that
\begin{multline}\label{SarnakGamma1}
\#\left\{f \in \BB_{\kappa}(\Gamma_1(q)) : t_f \in [0,T], \ \left|\lambda_f(p)\right| \geq \alpha_p \text{ for all $p \in \PP$}\right\}	\\
\ll_{\e} \vol(\Gamma_1(q) \backslash \Hb)^{1 - 3 \sum_{p \in \PP} \mu_p \frac{\log \alpha_p/2}{\log p} + \e} \left(T^2\right)^{1 - 4 \sum_{p \in \PP} \mu_p \frac{\log \alpha_p/2}{\log p} + \e},
\end{multline}
\begin{multline}\label{SarnakGamma}
\#\left\{f \in \BB_{\kappa}(\Gamma(q)) : t_f \in [0,T], \ \left|\lambda_f(p)\right| \geq \alpha_p \text{ for all $p \in \PP$}\right\}	\\
\ll_{\e} \vol(\Gamma(q) \backslash \Hb)^{1 - \frac{8}{3} \sum_{p \in \PP} \mu_p \frac{\log \alpha_p/2}{\log p} + \e} \left(T^2\right)^{1 - 4 \sum_{p \in \PP} \mu_p \frac{\log \alpha_p/2}{\log p} + \e},
\end{multline}
\begin{multline}\label{SarnakGamma0}
\#\left\{f \in \BB_{\kappa}(q,\chi) : t_f \in [0,T], \ \left|\lambda_f(p)\right| \geq \alpha_p \text{ for all $p \in \PP$}\right\}	\\
\ll_{\e} \vol(\Gamma_0(q) \backslash \Hb)^{1 - 4 \sum_{p \in \PP} \mu_p \frac{\log \alpha_p/2}{\log p} + \e} \left(T^2\right)^{1 - 4 \sum_{p \in \PP} \mu_p \frac{\log \alpha_p/2}{\log p} + \e}	\\
\times \min\left\{\dot{Q}^{4 \sum_{p \in \PP} \mu_p \frac{\log \alpha_p/2}{\log p}}, \ddot{Q}^{1 - 4 \sum_{p \in \PP} \mu_p \frac{\log \alpha_p/2}{\log p}}\right\}.
\end{multline}
\end{theorem}

\hyperref[Sarnakthm]{Theorem \ref*{Sarnakthm}} should be compared to the Weyl law, which states that
\[\#\left\{f \in \BB_{\kappa}(\Gamma,\chi) : t_f \in [0,T]\right\} \sim \frac{\vol\left(\Gamma \backslash \Hb\right)}{4\pi} T^2.\]

For $\Gamma = \SL_2(\Z)$, so that $\chi$ is the trivial character, and $\PP$ consisting of a single prime $p$, \hyperref[Sarnakthm]{Theorem \ref*{Sarnakthm}} is a result of Blomer, Buttcane, and Raulf \cite[Proposition 1]{BBR}, improving on a slightly weaker result of Sarnak \cite[Theorem 1.1]{Sarnak}, who uses the Selberg trace formula in place of the Kuznetsov formula and obtains instead (see \cite[Footnote 1]{BBR})
\[\#\left\{f \in \BB_0\left(\SL_2(\Z)\right) : t_f \in [0,T], \ \left|\lambda_f(p)\right| \geq \alpha\right\} \ll \left(T^2\right)^{1 - 2 \frac{\log \alpha / 2}{\log p}}.\]

\begin{theorem}\label{Huxleythm}
For any fixed finite (possibly empty) collection of primes $\PP$ not dividing $q$, any $\alpha_0 \in (0,1/2)$, $\alpha_p  \in (2, p^{1/2} + p^{-1/2})$, and $0 \leq \mu_0,\mu_p \leq 1$ for all $p \in \PP$ with $\mu_0 + \sum_{p \in \PP} \mu_p = 1$, we have that
\begin{multline}\label{HuxleyGamma1}
\#\left\{f \in \BB_0(\Gamma_1(q)) : it_f \in (\alpha_0,1/2), \ \left|\lambda_f(p)\right| \geq \alpha_p \text{ for all $p \in \PP$}\right\}	\\
\ll_{\e} \vol(\Gamma_1(q) \backslash \Hb)^{1 - 3 \left(\mu_0 \alpha_0 + \sum_{p \in \PP} \mu_p \frac{\log \alpha_p/2}{\log p}\right) + \e}
\end{multline}
\begin{multline}\label{HuxleyGamma}
\#\left\{f \in \BB_0(\Gamma(q)) : it_f \in (\alpha_0,1/2), \ \left|\lambda_f(p)\right| \geq \alpha_p \text{ for all $p \in \PP$}\right\}	\\
\ll_{\e} \vol(\Gamma(q) \backslash \Hb)^{1 - \frac{8}{3} \left(\mu_0 \alpha_0 + \sum_{p \in \PP} \mu_p \frac{\log \alpha_p/2}{\log p}\right) + \e}.
\end{multline}
\begin{multline}\label{HuxleyGamma0}
\#\left\{f \in \BB_0(q,\chi) : it_f \in (\alpha_0,1/2), \ \left|\lambda_f(p)\right| \geq \alpha_p \text{ for all $p \in \PP$}\right\}	\\
\ll_{\e} \vol(\Gamma_0(q) \backslash \Hb)^{1 - 4 \left(\mu_0 \alpha_0 + \sum_{p \in \PP} \mu_p \frac{\log \alpha_p/2}{\log p}\right) + \e}	\\
\times \min\left\{\dot{Q}^{4 \left(\mu_0 \alpha_0 + \sum_{p \in \PP} \mu_p \frac{\log \alpha_p/2}{\log p}\right)}, \ddot{Q}^{1 - 4 \left(\mu_0 \alpha_0 + \sum_{p \in \PP} \mu_p \frac{\log \alpha_p/2}{\log p}\right)}\right\}.
\end{multline}
\end{theorem}

When $\PP$ is empty and $\chi$ is the trivial congruence character, \hyperref[Huxleythm]{Theorem \ref*{Huxleythm}} improves upon a result of Huxley \cite{Hux}, who uses the Selberg trace formula in place of the Kuznetsov formula and obtains instead this result with the exponent $2$ for each of the three congruence subgroups instead of $3$, $8/3$, and $4$ respectively. When $\PP$ is empty and $\chi$ is the trivial congruence character, \eqref{HuxleyGamma0} is a result of Iwaniec \cite[Theorem 11.7]{Iwa} ; see also \cite[(16.61)]{IK}.

Since
\[\lfloor \frac{3\alpha + 1}{4} \rfloor - \frac{\alpha}{2} \leq \frac{3 \alpha}{10},\]
so that $\dot{Q} \ll \vol(\Gamma_0(q) \backslash \Hb)^{3/10}$, the right-hand side of \eqref{SarnakGamma0} is bounded by
\[\vol(\Gamma_0(q) \backslash \Hb)^{1 - \frac{14}{5} \sum_{p \in \PP} \mu_p \frac{\log \alpha_p/2}{\log p} + \e} \left(T^2\right)^{1 - 4 \sum_{p \in \PP} \mu_p \frac{\log \alpha_p/2}{\log p} + \e},\]
while the right-hand side of \eqref{HuxleyGamma0} is bounded by
\[\vol(\Gamma_0(q) \backslash \Hb)^{1 - \frac{14}{5} \left(\mu_0 \alpha_0 + \sum_{p \in \PP} \mu_p \frac{\log \alpha_p/2}{\log p}\right) + \e}.\]
On the other hand, taking $\PP$ to consist of a single prime in \eqref{SarnakGamma0} recovers the Selberg bound $\lambda_f(p) \ll_{\e} p^{1/4 + \e}$ for an individual element $f \in \BB_{\kappa}(q,\chi)$ by taking $T$ sufficiently large, while taking $\PP$ to be empty in \eqref{HuxleyGamma0} recovers the Selberg bound $\lambda_f \geq 3/16$ by embedding $f$ in $\BB_{\kappa}(qQ,\chi)$ and taking $Q$ sufficiently large.

Finally, we also prove the following improvements of \hyperref[Sarnakthm]{Theorems \ref*{Sarnakthm}} and \ref{Huxleythm} for $\Gamma_1(q)$ with $q$ squarefree via a twisting argument.

\begin{theorem}\label{improvedthm}
When $q$ is squarefree, \eqref{SarnakGamma1} and \eqref{HuxleyGamma1} hold with the exponent $3$ replaced by $4$.
\end{theorem}

\subsection{Idea of Proof}

By Rankin's trick (which is to say Chebyshev's inequality), it suffices to find bounds for
\[\sum_{\substack{f \in \BB_{\kappa}\left(\Gamma, \chi\right) \\ t_f \in [0,T]}} \prod_{p \in \PP} \left|\lambda_f(p)\right|^{2\ell_p}, \qquad \sum_{\substack{f \in \BB_0\left(\Gamma, \chi\right) \\ it_f \in (0,1/2)}} X^{2it_f} \prod_{p \in \PP} \left|\lambda_f(p)\right|^{2\ell_p}\]
for nonnegative integers $\ell_p$ and a positive real number $X \geq 1$ to be chosen. To bound these quantities, we begin with the Kuznetsov formula for $\BB_{\kappa}(q,\chi)$; we then use the Atkin--Lehner decomposition to turn this into a Kuznetsov formula for $\BB_{\kappa}(\Gamma,\chi)$. We take a test function in the Kuznetsov formula that localises the spectral sum to cusp forms with $t_f \in [0,T]$ in the case of \hyperref[Sarnakthm]{Theorem \ref*{Sarnakthm}} and to cusp forms with $it_f \in (0,1/2)$ in the case of \hyperref[Huxleythm]{Theorem \ref*{Huxleythm}}. We use the Hecke relations to introduce powers of the Hecke eigenvalues into the Kuznetsov formula. By positivity, we discard the contribution of the continuous spectrum, and we are left with bounding the right-hand side of the Kuznetsov formula.

The chief novelty of the proof is the bounds for sums of Kloosterman sums in the Kuznetsov formula for each congruence subgroup. As well as the usual Weil bound, we use character orthogonality for $\Gamma_1(q)$ and $\Gamma(q)$, at which point we only use the trivial bound for the resulting sum of Kloosterman sums. For $\Gamma_0(q)$ and $\chi$ the principal character, we may also use the Weil bound, but for $\chi$ nonprincipal, additional difficulties arise in bounding the Kloosterman sum, with the bound possibly depending on the conductor of $\chi$; it is for this reason that the bounds \eqref{SarnakGamma0} and \eqref{HuxleyGamma0} involve $\dot{Q}$, for $\dot{Q}$ arises when only weaker bounds than the Weil bound are possible for the Kloosterman sums involved.

We also highlight the key trick to proving \hyperref[improvedthm]{Theorem \ref*{improvedthm}}, namely that the Laplacian eigenvalue and absolute value of a Hecke eigenvalue of a Maa\ss{} form remain unchanged under twisting by a Dirichlet character. Twisting may alter the level of a Maa\ss{} form, yet \hyperref[improvedthm]{Theorem \ref*{improvedthm}} involves a favourable situation in which the resulting family of twisted Maa\ss{} forms are sufficiently well-behaved that we are able to improve the exponent in the density theorem.

It is worth mentioning that the results in this paper ought to generalise naturally to cusp forms on $\GL_2$ over arbitrary number fields $F$. In \cite{BrMia}, Bruggeman and Miatello prove a form of the Kuznetsov formula for $\GL_2$ over a totally real field and use this to prove weighted Weyl law for cusp forms. Similarly, in \cite{Mag}, Maga proves a semi-ad\`{e}lic version of the Kuznetsov formula for $\GL_2$ over an arbitrary number field. In the former case, this formula is valid for congruence subgroups of the form $\Gamma_0(\qq)$ for a nonzero integral ideal $\qq$ of the ring of integers $\OO_F$ of $F$ and arbitrary congruence characters $\chi$ modulo $\qq$, while the latter only treats the case of trivial congruence character but should easily be able to be generalised to arbitrary congruence character; this is precisely what is required for density theorems for the congruence subgroups $\Gamma_0(\qq)$, $\Gamma_1(\qq)$, and $\Gamma(\qq)$.

\section{The Kuznetsov Formula}

The background on automorphic forms and notation in this section largely follows \cite{DFI}; see \cite[Section 4]{DFI} for more details. Let $\kappa \in \{0,1\}$, and let $\chi$ be a primitive Dirichlet character modulo $q_{\chi}$, where $q_{\chi}$ divides $q$, satisfying $\chi(-1) = (-1)^{\kappa}$; this defines a congruence character of $\Gamma_0(q)$ via $\chi(\gamma) \defeq \chi(d)$ for $\gamma = \left(\begin{smallmatrix} a & b \\ c & d \end{smallmatrix}\right) \in \Gamma_0(q)$. We denote by $L^2\left(\Gamma_0(q) \backslash \Hb,\kappa,\chi\right)$ the $L^2$-completion of the space of all smooth functions $f : \Hb \to \C$ that are of moderate growth and satisfy $f(\gamma z) = \chi(\gamma) j_{\gamma}(z)^{\kappa} f(z)$. This space has the spectral decomposition
\[L^2\left(\Gamma_0(q) \backslash \Hb,\kappa,\chi\right) = \AA_{\kappa}(q,\chi) \oplus \EE_{\kappa}(q,\chi)\]
with respect to the weight $\kappa$ Laplacian, where $\AA_{\kappa}(q,\chi) \defeq \AA_{\kappa}\left(\Gamma_0(q),\chi\right)$ is the space spanned by Maa\ss{} cusp forms of weight $\kappa$, level $q$, and nebentypus $\chi$, and $\EE_{\kappa}(q,\chi)$ is the space spanned by incomplete Eisenstein series parametrised by the cusps $\aa$ of $\Gamma_0(q) \backslash \Hb$ that are singular with respect to $\chi$.

We denote by $\BB_{\kappa}(q,\chi)$ an orthonormal basis of Maa\ss{} cusp forms $f \in \AA_{\kappa}(q,\chi)$ normalised to have $L^2$-norm $1$:
\[\langle f, f \rangle_q \defeq \int_{\Gamma_0(q) \backslash \Hb} |f(z)|^2 \, d\mu(z) = 1,\]
where $d\mu(z) = \dfrac{dx \, dy}{y^2}$ is the $\SL_2(\R)$-invariant measure on $\Hb$. Later we will use the Atkin--Lehner decomposition of $\AA_{\kappa}(q,\chi)$ in order to specify that $\BB_{\kappa}(q,\chi)$ can be chosen to consist of linear combinations of Hecke eigenforms. The Fourier expansion of $f \in \BB_{\kappa}(q,\chi)$ is
\[f(z) = \sum_{\substack{n = -\infty \\ n \neq 0}}^{\infty} \rho_f(n) W_{\sgn(n) \frac{\kappa}{2}, it_f}(4\pi|n|y) e(nx),\]
where $W_{\alpha,\beta}$ is the Whittaker function and
\[\rho_f(n) W_{\sgn(n) \frac{\kappa}{2}, it_f}(4\pi|n|y) = \int_{0}^{1} f(z) e(-nx) \, dx.\]

For a singular cusp $\aa$, we define the Eisenstein series
\[E_{\aa}(z,s,\chi) \defeq \sum_{\gamma \in \Gamma_{\aa} \backslash \Gamma_0(q)} \overline{\chi}(\gamma) j_{\sigma_{\aa}^{-1} \gamma}(z)^{-\kappa} \Im\left(\sigma_{\aa}^{-1} \gamma z\right)^s,\]
which is absolutely convergent for $\Re(s) > 1$ and extends meromorphically to $\C$, with the Fourier expansion
\[\delta_{\aa,\infty} y^{1/2 + it} + \varphi_{\aa,\infty}\left(\frac{1}{2} + it,\chi\right) y^{1/2 - it} + \sum_{\substack{n = -\infty \\ n \neq 0}}^{\infty} \rho_{\aa}(n,t,\chi) W_{\sgn(n) \frac{\kappa}{2}, it}(4\pi|n|y) e(nx)\]
for $s = 1/2 + it$ with $t \in \R \setminus \{0\}$, where
\begin{align*}
\delta_{\aa,\infty} y^{1/2 + it} + \varphi_{\aa,\infty}\left(\frac{1}{2} + it,\chi\right) y^{1/2 - it} & \defeq \int_{0}^{1} E_{\aa}\left(z, \frac{1}{2} + it, \chi\right) \, dx,	\\
\rho_{\aa}(n,t,\chi) W_{\sgn(n) \frac{\kappa}{2}, it}(4\pi|n|y) & \defeq \int_{0}^{1} E_{\aa}\left(z, \frac{1}{2} + it, \chi\right) e(-nx) \, dx.
\end{align*}
The subspace $\EE_{\kappa}(q,\chi)$ consists of functions $g \in L^2\left(\Gamma_0(q) \backslash \Hb, \kappa, \chi\right)$ that are orthogonal to every Maa\ss{} cusp form $f \in \AA_{\kappa}(q,\chi)$; it is the $L^2$-closure of the space spanned by incomplete Eisenstein series, which are functions of the form
\begin{equation}\label{incompleteEiseneq}
E_{\aa}(z,\psi,\chi) \defeq \frac{1}{2\pi i} \int_{\sigma - i\infty}^{\sigma + i\infty} E_{\aa}(z,s,\chi) \widehat{\psi}(s) \, ds 
\end{equation}
for some singular cusp $\aa$ and some smooth function of compact support $\psi : \R^+ \to \C$, where $\sigma > 1$ and
\[\widehat{\psi}(s) \defeq \int_{0}^{\infty} \psi(x) x^{-s} \, \frac{dx}{x}.\]

\begin{theorem}[{\cite[Proposition 5.2]{DFI}}]
For $m,n \geq 1$ and $r \in \R$,
\begin{multline*}
\sum_{f \in \BB_{\kappa}(q,\chi)} \frac{4\pi \sqrt{mn} \overline{\rho_f}(m) \rho_f(n)}{\cosh \pi (r - t_f) \cosh \pi (r + t_f)} + \sum_{\aa} \int_{-\infty}^{\infty} \frac{\sqrt{mn} \overline{\rho_{\aa}}(m,t,\chi) \rho_{\aa}(n,t,\chi)}{\cosh \pi (r - t) \cosh \pi (r + t)} \, dt	\\
= \frac{\left|\Gamma\left(1 - \frac{\kappa}{2} - ir\right)\right|^2}{\pi^2} \left(\delta_{m,n} + \sum_{\substack{c = 1 \\ c \equiv 0 \hspace{-.25cm} \pmod{q}}}^{\infty} \frac{S_{\chi}(m,n;c)}{c} I_{\kappa}\left(\frac{4\pi \sqrt{mn}}{c},r\right)\right),
\end{multline*}
where
\begin{align*}
S_{\chi}(m,n;c) & \defeq \sum_{d \in (\Z/c\Z)^{\times}} \chi(d) e\left(\frac{md + n\overline{d}}{c}\right),	\\
I_{\kappa}(t,r) & \defeq -2t \int_{-i}^{i} (-i\zeta)^{\kappa - 1} K_{2ir}(\zeta t) \, d\zeta,
\end{align*}
with the latter integral being over the semicircle $|z| = 1$, $\Re(z) > 0$. 
\end{theorem}

By the reflection formula for the gamma function, we have that for $r \in \R$,
\[\left|\Gamma\left(1 - \frac{\kappa}{2} - ir\right)\right|^2 = \begin{dcases*}
\frac{\pi r}{\sinh \pi r} & if $\kappa = 0$,	\\
\frac{\pi}{\cosh \pi r} & if $\kappa = 1$.
\end{dcases*}\]

Given a sufficiently well-behaved function $h$, we may multiply both sides of the pre-Kuznetsov formula for $\kappa = 0$ by
\[\frac{1}{2} \left(h\left(r + \frac{i}{2}\right) + h\left(r - \frac{i}{2}\right)\right) \cosh \pi r\]
and then integrate both sides from $-\infty$ to $\infty$ with respect to $r$. This yields the following Kuznetsov formula.

\begin{theorem}[{see \cite[Section 2.1.4]{BHM}, \cite[Theorem 16.3]{IK}, \cite[Equation (7.32)]{KL}}]
Let $\delta > 0$, and let $h$ be a function that is even, holomorphic in the horizontal strip $|\Im(t)| \leq 1/2 + \delta$, and satisfies $h(t) \ll (|t| + 1)^{-2 - \delta}$. Then
\begin{multline*}
\sum_{f \in \BB_0(q,\chi)} 4\pi \sqrt{mn} \overline{\rho_f}(m) \rho_f(n) \frac{h(t_f)}{\cosh \pi t_f}	\\
+ \sum_{\aa} \int_{-\infty}^{\infty} \sqrt{mn} \overline{\rho_{\aa}}(m,t,\chi) \rho_{\aa}(n,t,\chi) \frac{h(t)}{\cosh \pi t} \, dt	\\
= \delta_{mn} g_0 + \sum_{\substack{c = 1 \\ c \equiv 0 \hspace{-.25cm} \pmod{q}}}^{\infty} \frac{S_{\chi}(m,n;c)}{c} g_0\left(\frac{4\pi \sqrt{mn}}{c}\right),
\end{multline*}
where
\begin{align*}
g_0 & \defeq \frac{1}{\pi} \int_{-\infty}^{\infty} r h(r) \tanh \pi r \, dr,	\\
g_0(x) & \defeq 2i \int_{-\infty}^{\infty} J_{2ir}(x) \frac{r h(r)}{\cosh \pi r} \, dr.
\end{align*}
\end{theorem}

The left-hand side of the Kuznetsov formula is called the spectral side; the first term is the contribution from the discrete spectrum, while the second term is the contribution from the continuous spectrum. The right-hand side of the Kuznetsov formula is called the geometric side; the first term is the delta term and the second term is the Kloosterman term.

\section{Decomposition of Spaces of Modular Forms}

\subsection{Eisenstein Series and Hecke Operators}

The space $\EE_{\kappa}(q,\chi)$ is spanned by incomplete Eisenstein series of the form \eqref{incompleteEiseneq}, which are obtained by integrating test functions against Eisenstein series indexed by singular cusps $\aa$; in this sense, the Eisenstein series $E_{\aa}(z,s,\chi)$ are a spanning set for $\EE_{\kappa}(q,\chi)$. We may instead choose a different spanning set of Eisenstein series for $\EE_{\kappa}(q,\chi)$; in place of the set of Eisenstein series $E_{\aa}(z,s,\chi)$ with $\aa$ a singular cusp, we may instead choose a spanning set of Eisenstein series of the form $E(z,s,f)$ with Fourier expansion
\[c_{1,f}(t) y^{1/2 + it} + c_{2,f}(t) y^{1/2 - it} + \sum_{\substack{n = -\infty \\ n \neq 0}}^{\infty} \rho_f(n,t,\chi) W_{\sgn(n) \frac{\kappa}{2}, it}(4\pi|n|y) e(nx)\]
for $s = 1/2 + it$ with $t \in \R \setminus \{0\}$, where $\BB(\chi_1,\chi_2) \ni f$ with $\chi_1 \chi_2 = \chi$ is some finite set depending on $\chi_1,\chi_2$ corresponding to an orthonormal basis in the space of the induced representation constructed out of the pair $(\chi_1,\chi_2)$; see \cite[Section 2.1.1]{BHM} or \cite[Chapter 5]{KL}. For our purposes, we need not be more specific about $\BB(\chi_1,\chi_2)$, other than noting that for each $f \in \BB(\chi_1,\chi_2)$, the Eisenstein series $E(z,1/2 + it,f)$ is an eigenfunction of the Hecke operators $T_n$ for $(n,q) = 1$ with Hecke eigenvalues
\[\lambda_f(n,t) = \sum_{ab = n} \chi_1(a) a^{it} \chi_2(b) b^{-it},\]
where for $g : \Hb \to \C$ a periodic function of period one,
\[(T_n g)(z) \defeq \frac{1}{\sqrt{n}} \sum_{ad = n} \chi(a) \sum_{b \hspace{-.25cm} \pmod{d}} g\left(\frac{az + b}{d}\right).\]
So for $f \in \BB(\chi_1,\chi_2)$,
\begin{align}
\lambda_f(m,t) \lambda_f(n,t) & = \sum_{d \mid (m,n)} \chi(d) \lambda_f\left(\frac{mn}{d^2},t\right),	\label{Eisensteinmult}\\
\overline{\lambda_f}(n,t) & = \overline{\chi}(n) \lambda_f(n,t),	\label{Eisensteinconj}\\
\rho_f(1,t) \lambda_f(n) & = \sqrt{n} \rho_f(n,t)	\label{Eisensteinrholambda}
\end{align}
whenever $m,n \geq 1$ with $(mn,q) = 1$ and $s = 1/2 + it$.

\begin{lemma}[{Cf.~\cite[Lemma 3]{CDF}, \cite[Lemma 2.8]{HM}, \cite[Section 6]{PY}}]
For any prime $p \nmid q$ and positive integer $\ell$, we have that
\begin{equation}\label{EisensteinHecke2ell}
\left|\lambda_f(p,t)\right|^{2\ell} = \sum_{j = 0}^{\ell} \alpha_{2j,2\ell} \overline{\chi}(p)^j \lambda_f\left(p^{2j},t\right)
\end{equation}
for any $f \in \BB(\chi_1,\chi_2)$ and $s = 1/2 + it$, where
\begin{equation}\label{Lobbnumber}
\alpha_{2j,2\ell} = \frac{2j + 1}{\ell + j + 1} \binom{2\ell}{\ell + j} = \begin{dcases*}
\binom{2\ell}{\ell - j} - \binom{2\ell}{\ell - j - 1} & if $0 \leq j \leq \ell - 1$,	\\
1 & if $j = \ell$,
\end{dcases*}
\end{equation}
so that each $\alpha_{2j,2\ell}$ is positive and satisfies
\begin{equation}\label{sumalphabound}
\sum_{j = 0}^{\ell} \alpha_{2j,2\ell} = \binom{2\ell}{\ell} \leq 2^{2\ell}.
\end{equation}
\end{lemma}

\begin{proof}
That \eqref{sumalphabound} follows from \eqref{Lobbnumber} is clear. For \eqref{EisensteinHecke2ell}, we have that
\[\overline{\chi}(p)^{j/2} \lambda_f\left(p^j,t\right) = U_j\left(\frac{\overline{\chi}(p)^{1/2} \lambda_f(p,t)}{2}\right),\]
where $U_j$ is the $j$-th Chebyshev polynomial of the second kind, because $U_j$ satisfies $U_0(x/2) = 1$, $U_1(x/2) = x$, and the recurrence relation
\[U_{j + 1}\left(\frac{x}{2}\right) = x U_j\left(\frac{x}{2}\right) - U_{j - 1}\left(\frac{x}{2}\right)\]
for all $j \geq 1$, and $\overline{\chi}(p)^{j/2} \lambda_f\left(p^j,t\right)$ satisfies the same recurrence relation from \eqref{Eisensteinmult}. Since
\[\frac{2}{\pi} \int_{-1}^{1} U_j(x) U_k(x) \sqrt{1 - x^2} \, dx = \delta_{j,k},\]
we have that
\[x^{2\ell} = \sum_{j = 0}^{2\ell} \alpha_{j,2\ell} U_j\left(\frac{x}{2}\right),\]
where
\[\alpha_{j,2\ell} = \frac{2^{2\ell + 1}}{\pi} \int_{-1}^{1} x^{2\ell} U_j(x) \sqrt{1 - x^2} \, dx.\]
This vanishes if $j$ is odd as $U_j(-x) = (-1)^j U_j(x)$, while for $j$ even we have the identity \eqref{Lobbnumber} from \cite[7.311.2]{GR}. Combined with \eqref{Eisensteinconj}, this proves \eqref{EisensteinHecke2ell}.
\end{proof}

\subsection{Atkin--Lehner Decomposition for \texorpdfstring{$\Gamma_0(q)$}{\textGamma\9040\200(q)}}

Similarly, we may choose a basis of $\AA_{\kappa}(q,\chi)$ consisting of linear combinations of Hecke eigenforms. Let $\BB_{\kappa}^{\ast}(q,\chi)$ denote the set of newforms of weight $\kappa$, level $q$, and nebentypus $\chi$, and let $\AA_{\kappa}^{\ast}(q,\chi)$ denote the subspace of $\AA_{\kappa}(q,\chi)$ spanned by such newforms. Recall that a newform $f \in \BB_{\kappa}^{\ast}(q,\chi)$ is an eigenfunction of the weight $\kappa$ Laplacian $\Delta_{\kappa}$ with eigenvalue $1/4 + t_f^2$ and of every Hecke operator $T_n$, $n \geq 1$, with eigenvalue $\lambda_f(n)$, as well as the operator $Q_{1/2 + it_f,\kappa}$ as defined in \cite[Section 4]{DFI}, with eigenvalue $\epsilon_f \in \{-1,1\}$; we say that $f$ is even if $\epsilon_f = 1$ and $f$ is odd if $\epsilon_f = -1$. In particular,
\begin{align}
\lambda_f(m) \lambda_f(n) & = \sum_{\substack{d \mid (m,n) \\ (d,q) = 1}} \chi(d) \lambda_f\left(\frac{mn}{d^2}\right),	\label{cuspmult}\\
\rho_f(1) \lambda_f(n) & = \sqrt{n} \rho_f(n) \label{cusprholambda}
\end{align}
whenever $m,n \geq 1$, and
\begin{equation}\label{cuspconj}
\overline{\lambda_f}(n) = \overline{\chi}(n) \lambda_f(n)
\end{equation}
for $n \geq 1$ with $(n,q) = 1$. Using \eqref{cuspmult} and \eqref{cuspconj}, we have the following.

\begin{lemma}
For any prime $p \nmid q$ and positive integer $\ell$, we have that
\begin{equation}\label{cuspHecke2ell}
\left|\lambda_f(p)\right|^{2\ell} = \sum_{j = 0}^{\ell} \alpha_{2j,2\ell} \overline{\chi}(p)^j \lambda_f\left(p^{2j}\right)
\end{equation}
for any $f \in \BB_{\kappa}^{\ast}(q,\chi)$, where once again $\alpha_{2j,2\ell}$ is given by \eqref{Lobbnumber}.
\end{lemma}

The Atkin--Lehner decomposition states that
\[\AA_{\kappa}(q,\chi) = \bigoplus_{\substack{q_1 q_2 = q \\ q_1 \equiv 0 \hspace{-.25cm} \pmod{q_{\chi}}}} \bigoplus_{f \in \BB_{\kappa}^{\ast}(q_1,\chi)} \bigoplus_{d \mid q_2} \C \cdot \iota_{d,q_1,q} f,\]
where $\iota_{d,q_1,q} : \AA_{\kappa}\left(q_1,\chi\right) \to \AA_{\kappa}(q,\chi)$ is the map $\iota_{d,q_1,q} f(z) = f(dz)$. The map $\iota_{d,q_1,q}$ commutes with the weight $k$ Laplacian $\Delta_{\kappa}$ and the Hecke operators $T_n$ whenever $n \geq 1$ and $(n,q) = 1$. It follows that if $g = \iota_{d,q_1,q} f$ for some $f \in \BB_{\kappa}^{\ast}\left(q_1,\chi\right)$, then $t_g = t_f$ and $\lambda_g(n) = \lambda_f(n)$ whenever $n \geq 1$ and $(n,q) = 1$. Note, however, that $\rho_g(1) = 0$ unless $d = 1$, in which case $\rho_g(1) = \rho_f(1)$.

Unfortunately, the inner Atkin--Lehner decomposition
\[\bigoplus_{d \mid q_2} \C \cdot \iota_{d,q_1,q} f\]
is not an orthogonal decomposition. Nonetheless, one may make use of this decomposition in determining an orthonormal basis of $\AA_{\kappa}(q,\chi)$. For squarefree $q$ and principal nebentypus, this is a result of Iwaniec, Luo, and Sarnak \cite[Lemma 2.4]{ILS}, while Blomer and Mili\'{c}evi\'{c} have generalised this to nonsquarefree $q$ \cite[Lemma 9]{BlMil}. Here we generalise this further to nonprincipal nebentypus; this has also independently been derived by Schulze-Pillot and Yenirce \cite{S-PY} via a different method.

\begin{lemma}[{Cf.~\cite[Lemma 2.4]{ILS}, \cite[Lemma 9]{BlMil}}]\label{innerproductlemma}
Suppose that $\chi$ has conductor $q_{\chi} \mid q$, and suppose that $q_1 q_2 = q$ with $q_1 \equiv 0 \pmod{q_{\chi}}$. For $f \in \BB_{\kappa}^{\ast}(q_1,\chi)$ and $\ell_1, \ell_2 \mid q_2$, we have that
\[\frac{\left\langle \iota_{\ell_1,q_1,q} f, \iota_{\ell_2,q_1,q} f\right\rangle_q}{\langle \iota_{1,q_1,q} f, \iota_{1,q_1,q} f\rangle_q} = A_f\left(\frac{\ell_2}{(\ell_1,\ell_2)}\right) \overline{A_f}\left(\frac{\ell_1}{(\ell_1,\ell_2)}\right),\]
where $A_f(n)$ is the multiplicative function defined on prime powers by
\[A_f(p^t) = \begin{dcases*}
\frac{\lambda_f(p)}{\sqrt{p}(1 + \chi_{0(q_1)}(p) p^{-1})} & if $t = 1$,	\\
\frac{\lambda_f(p^t) - \chi_{(q_1)}(p) \lambda_f(p^{t - 2}) p^{-1}}{p^{t/2} (1 + \chi_{0(q_1)}(p) p^{-1})} & if $t \geq 2$,
\end{dcases*}\]
where $\chi_{0(q_1)}$ denotes the principal character modulo $q_1$ and $\chi_{(q_1)} \defeq \chi \chi_{0(q_1)}$ denotes the Dirichlet character modulo $q_1$ induced from $\chi$.
\end{lemma}

\begin{proof}
For $\Re(s) > 1$, consider the integral
\[F(s) \defeq \int_{\Gamma_0(q) \backslash \Hb} f(\ell_1 z) \overline{f}(\ell_2 z) E(z,s) \, d\mu(z),\]
where
\[E(z,s) \defeq \sum_{\gamma \in \Gamma_{\infty} \backslash \Gamma_0(q)} \Im(\gamma z)^s.\]
Unfolding the integral and using Parseval's identity,
\[F(s) = \int_{0}^{\infty} y^{s - 1} \underset{\ell_1 n_1 = \ell_2 n_2}{\sum_{\substack{n_1 = -\infty \\ n_1 \neq 0}}^{\infty} \sum_{\substack{n_2 = -\infty \\ n_2 \neq 0}}^{\infty}} \rho_f(n_1) \overline{\rho_f}(n_2) W_{\sgn(n_1) \frac{\kappa}{2}, it_f}(4\pi \ell_1 |n_1|y)^2 \, \frac{dy}{y}.\]
From \eqref{cusprholambda} and the fact from \cite[Equation (4.70)]{DFI} that
\[\rho_f(-n) = \epsilon_f \frac{\Gamma\left(\frac{1 + \kappa}{2} + it_f\right)}{\Gamma\left(\frac{1 - \kappa}{2} + it_f\right)} \rho_f(n)\]
for $n \geq 1$, where $\epsilon_f \in \{-1,1\}$, we find that
\begin{multline*}
F(s) = \frac{|\rho_f(1)|^2}{(4\pi [\ell_1, \ell_2])^{s - 1} \sqrt{\ell' \ell''}} \sum_{n = 1}^{\infty} \frac{\lambda_f(\ell'' n) \overline{\lambda_f}(\ell' n)}{n^s}	\\
\times \int_{0}^{\infty} y^{s - 1} \left(W_{\frac{\kappa}{2}, it_f}(y)^2 + \left|\frac{\Gamma\left(\frac{1 + \kappa}{2} + it_f\right)}{\Gamma\left(\frac{1 - \kappa}{2} + it_f\right)}\right|^2 W_{-\frac{\kappa}{2}, it_f}(y)^2\right) \, \frac{dy}{y},
\end{multline*}
where we have written $n_1 = \ell'' n$, $n_2 = \ell' n$, with $\ell' = \ell_1/(\ell_1,\ell_2)$ and $\ell'' = \ell_2/(\ell_1,\ell_2)$.

Next, by the multiplicativity of the Hecke eigenvalues of $f$ together with the fact that $(\ell',\ell'') = 1$, the sum over $n$ is equal to
\[\sum_{\substack{n = 1 \\ (n,\ell' \ell'') = 1}}^{\infty} \frac{|\lambda_f(n)|^2}{n^s} \prod_{p^t \parallel \ell''} \sum_{r = 0}^{\infty} \frac{\lambda_f(p^{r + t}) \overline{\lambda_f}(p^r)}{p^{rs}} \prod_{p^t \parallel \ell'} \sum_{r = 0}^{\infty} \frac{\lambda_f(p^r) \overline{\lambda_f}(p^{r + t})}{p^{rs}}.\]
From \eqref{cuspmult} and \eqref{cuspconj}, we find that
\begin{align*}
\sum_{r = 0}^{\infty} \frac{\lambda_f(p^{r + t}) \overline{\lambda_f}(p^r)}{p^{rs}} & = B_f(p^t;s) \sum_{r = 0}^{\infty} \frac{|\lambda_f(p^r)|^2}{p^{rs}},	\\
\sum_{r = 0}^{\infty} \frac{\lambda_f(p^r) \overline{\lambda_f}(p^{r + t})}{p^{rs}} & = \overline{B_f}(p^t;\overline{s}) \sum_{r = 0}^{\infty} \frac{|\lambda_f(p^r)|^2}{p^{rs}},
\end{align*}
where $B_f(n;s)$ is defined to be the multiplicative function
\[B_f(p^t;s) = \begin{dcases*}
\frac{\lambda_f(p)}{1 + \chi_{0(q_1)}(p) p^{-s}} & if $t = 1$,	\\
\frac{\lambda_f(p^t) - \chi_{(q_1)}(p) \lambda_f(p^{t - 2}) p^{-s}}{1 + \chi_{0(q_1)}(p) p^{-s}} & if $t \geq 2$,
\end{dcases*}\]
so that $A_f(n) = n^{-1/2} B_f(n;1)$. We surmise that $F(s)$ is equal to
\begin{multline}\label{F(s)}
\frac{|\rho_f(1)|^2}{(4\pi [\ell_1, \ell_2])^{s - 1} \sqrt{\ell' \ell''}} B_f(\ell'';s) \overline{B_f}(\ell';\overline{s}) \sum_{n = 1}^{\infty} \frac{|\lambda_f(n)|^2}{n^s}	\\
\times \int_{0}^{\infty} y^{s - 1} \left(W_{\frac{\kappa}{2}, it_f}(y)^2 + \left|\frac{\Gamma\left(\frac{1 + \kappa}{2} + it_f\right)}{\Gamma\left(\frac{1 - \kappa}{2} + it_f\right)}\right|^2 W_{-\frac{\kappa}{2}, it_f}(y)^2\right) \, \frac{dy}{y}.
\end{multline}
The result follows by taking the residue at $s = 1$, noting that $E(z,s)$ has residue equal to $1/\vol(\Gamma_0(q) \backslash \Hb)$ at $s = 1$ independently of $z \in \Gamma_0(q) \backslash \Hb$, and comparing to the case $\ell_1 = \ell_2 = 1$.
\end{proof}

\begin{lemma}[{Cf.~\cite[Lemma 9]{BlMil}}]
An orthonormal basis of $\AA_{\kappa}(q,\chi)$ is given by
\begin{equation}\label{basisqchi}
\BB_{\kappa}(q,\chi) = \bigsqcup_{\substack{q_1 q_2 = q \\ q_1 \equiv 0 \hspace{-.25cm} \pmod{q_{\chi}}}} \bigsqcup_{f \in \BB_{\kappa}^{\ast}(q_1,\chi)} \bigsqcup_{d \mid q_2} \left\{f_d = \sum_{\ell \mid d} \xi_f(\ell,d) \iota_{\ell,q_1,q} f\right\},
\end{equation}
where each $f \in \BB_{\kappa}^{\ast}\left(q_1,\chi\right)$ is normalised such that $\langle \iota_{1,q_1,q} f, \iota_{1,q_1,q} f \rangle_q = 1$, and the function $\xi_f(\ell,d)$ is jointly multiplicative. For $0 \leq r \leq t$, $\xi_f(p^r,p^t)$ is equal to
\[\begin{dcases*}
1 & if $r = t = 0$,	\\
-\frac{\overline{A_f}(p)}{\sqrt{1 - |A_f(p)|^2}} & if $r = 0$ and $t = 1$,	\\
\frac{1}{\sqrt{1 - |A_f(p)|^2}} & if $r = t = 1$,	\\
\frac{\overline{\chi}_{(q_1)}(p)}{p} \frac{1}{\sqrt{(1 - \chi_{0(q_1)}(p) p^{-2}) (1 - |A_f(p)|^2)}} & if $r = t - 2$ and $t \geq 2$,	\\
-\frac{\overline{\lambda_f}(p)}{\sqrt{p}} \frac{1}{\sqrt{(1 - \chi_{0(q_1)}(p) p^{-2}) (1 - |A_f(p)|^2)}} & if $r = t - 1$ and $t \geq 2$,	\\
\frac{1}{\sqrt{(1 - \chi_{0(q_1)}(p) p^{-2}) (1 - |A_f(p)|^2)}} & if $r = t$ and $t \geq 2$,	\\
0 & if $0 \leq r \leq t - 3$ and $t \geq 3$.
\end{dcases*}\]
\end{lemma}

The key point is that the coefficients $\xi_f(\ell,d)$ are chosen such that the ratio of inner products
\[\delta_f(d_1,d_2) \defeq \frac{\langle f_{d_1}, f_{d_2}\rangle_q}{\langle \iota_{1,q_1,q} f, \iota_{1,q_1,q} f\rangle_q} = \sum_{\ell_1 \mid d_1} \sum_{\ell_2 \mid d_2} \xi_f(\ell_1,d_1) \overline{\xi_f}(\ell_2,d_2) \frac{\left\langle \iota_{\ell_1,q_1,q} f, \iota_{\ell_2,q_1,q} f\right\rangle_q}{\langle \iota_{1,q_1,q} f, \iota_{1,q_1,q} f\rangle_q}\]
is equal to $1$ if $d_1 = d_2$ and $0$ otherwise.

\begin{proof}
The proof follows the same lines as \cite[Proof of Lemma 9]{BlMil}; we omit the details.
\end{proof}

\subsection{Explicit Kuznetsov Formula}

We may use the explicit basis \eqref{basisqchi} together with \eqref{cuspconj} and \eqref{cusprholambda} to rewrite the discrete part of the Kuznetsov formula, noting that for $f \in \BB_{\kappa}^{\ast}(q_1,\chi)$, $d \mid q_2$, and $n \geq 1$ coprime to $q$,
\[\rho_{f_d}(n) = \xi_f(1,d) \rho_f(1) \frac{\lambda_f(n)}{\sqrt{n}}.\]
Similarly, the continuous part can be rewritten in terms of the Eisenstein spanning set $\BB(\chi_1,\chi_2)$ with $\chi_1 \chi_2 = \chi$ together with \eqref{Eisensteinconj} and \eqref{Eisensteinrholambda}. This yields the following explicit versions of the pre-Kuznetsov and Kuznetsov formul\ae{}.

\begin{proposition}
When $m,n \geq 1$ with $(mn,q) = 1$, the pre-Kuznetsov formula has the form
\begin{multline}\label{pre-Kuznetsov}
\sum_{\substack{q_1 q_2 = q \\ q_1 \equiv 0 \hspace{-.25cm} \pmod{q_{\chi}}}} \sum_{f \in \BB_{\kappa}^{\ast}\left(q_1,\chi\right)} 4 \pi \xi_f \left|\rho_f(1)\right|^2 \frac{\overline{\chi}(m) \lambda_f(m) \lambda_f(n)}{\cosh \pi(r - t_f) \cosh \pi (r + t_f)}	\\
+ \sum_{\substack{\chi_1,\chi_2 \hspace{-.25cm} \pmod{q} \\ \chi_1 \chi_2 = \chi}} \sum_{f \in \BB(\chi_1,\chi_2)} \int_{-\infty}^{\infty} \left|\rho_f(1,t)\right|^2 \frac{\overline{\chi}(m) \lambda_f(m,t) \lambda_f(n,t)}{\cosh \pi(r - t) \cosh \pi (r + t)} \, dt	\\
= \frac{\left|\Gamma\left(1 - \frac{\kappa}{2} - ir\right)\right|^2}{\pi^2} \left(\delta_{mn} + \sum_{\substack{c = 1 \\ c \equiv 0 \hspace{-.25cm} \pmod{q}}}^{\infty} \frac{S_{\chi}(m,n;c)}{c} I_{\kappa}\left(\frac{4\pi \sqrt{mn}}{c},r\right)\right)
\end{multline}
for $\kappa \in \{0,1\}$, where we define
\[\xi_f \defeq \sum_{d \mid q_2} \left|\xi_f(1,d)\right|^2,\]
while the Kuznetsov formula for $\kappa = 0$ has the form
\begin{multline}\label{Kuznetsov}
\sum_{\substack{q_1 q_2 = q \\ q_1 \equiv 0 \hspace{-.25cm} \pmod{q_{\chi}}}} \sum_{f \in \BB_0^{\ast}\left(q_1,\chi\right)} \frac{4\pi \xi_f \left|\rho_f(1)\right|^2}{\cosh \pi t_f} \overline{\chi}(m) \lambda_f(m) \lambda_f(n) h(t_f)	\\
+ \sum_{\substack{\chi_1,\chi_2 \hspace{-.25cm} \pmod{q} \\ \chi_1 \chi_2 = \chi}} \sum_{f \in \BB(\chi_1,\chi_2)} \int_{-\infty}^{\infty} \frac{\left|\rho_f(1,t)\right|^2}{\cosh \pi t} \overline{\chi}(m) \lambda_f(m,t) \lambda_f(n,t) h(t) \, dt	\\
= \delta_{mn} g_0 + \sum_{\substack{c = 1 \\ c \equiv 0 \hspace{-.25cm} \pmod{q}}}^{\infty} \frac{S_{\chi}(m,n;c)}{c} g_0\left(\frac{4\pi \sqrt{mn}}{c}\right).
\end{multline}
In both formul\ae{}, each $f \in \BB_{\kappa}^{\ast}\left(q_1,\chi\right)$ is normalised such that $\langle \iota_{1,q_1,q} f, \iota_{1,q_1,q} f \rangle_q = 1$.
\end{proposition}

\subsection{Atkin--Lehner Decomposition for \texorpdfstring{$\Gamma_1(q)$}{\textGamma\9040\201(q)}}

We recall the decomposition
\[\AA_{\kappa}\left(\Gamma_1(q)\right) = \bigoplus_{\substack{\chi \hspace{-.25cm}\pmod{q} \\ \chi(-1) = (-1)^{\kappa}}} \AA_{\kappa}(q,\chi),\]
which follows from the fact that $\Gamma_1(q)$ is a normal subgroup of $\Gamma_0(q)$ with quotient group isomorphic to $(\Z/q\Z)^{\times}$, noting that $\AA_{\kappa}(q,\chi) = \{0\}$ if $\chi(-1) \neq (-1)^{\kappa}$. From this, we obtain the natural basis of $\AA_{\kappa}\left(\Gamma_1(q)\right)$ given by
\begin{equation}\label{basisGamma1(q)}
\BB_{\kappa}\left(\Gamma_1(q)\right) = \bigsqcup_{\substack{\chi \hspace{-.25cm} \pmod{q} \\ \chi(-1) = (-1)^{\kappa}}} \bigsqcup_{\substack{q_1 q_2 = q \\ q_1 \equiv 0 \hspace{-.25cm} \pmod{q_{\chi}}}} \bigsqcup_{f \in \BB_{\kappa}^{\ast}(q_1,\chi)} \bigsqcup_{d \mid q_2} \left\{f_d = \sum_{\ell \mid d} \xi_f(\ell,d) \iota_{\ell,q_1,q} f\right\}.
\end{equation}
This allows us to use the pre-Kuznetsov and Kuznetsov formul\ae{} \eqref{pre-Kuznetsov} and \eqref{Kuznetsov} for $\BB_{\kappa}(\Gamma_1(q))$ and $\BB_0(\Gamma_1(q))$, even though ostensibly these two formul\ae{} are only set up for $\BB_{\kappa}(q,\chi)$ and $\BB_0(q,\chi)$.

\subsection{Atkin--Lehner Decomposition for \texorpdfstring{$\Gamma(q)$}{\textGamma(q)}}

A similar decomposition also holds for $\AA_{\kappa}(\Gamma(q))$. In this case, the fact that
\begin{align*}
\Gamma_0\left(q^2\right) \cap \Gamma_1(q) & = \left\{\begin{pmatrix} a & b \\ c & d \end{pmatrix} \in \SL_2(\Z) : a,d \equiv 1 \hspace{-.25cm} \pmod{q}, \ c \equiv 0 \hspace{-.25cm} \pmod{q^2}\right\}	\\
& = \begin{pmatrix} q^{-1} & 0 \\ 0 & 1 \end{pmatrix} \Gamma(q) \begin{pmatrix} q & 0 \\ 0 & 1 \end{pmatrix}
\end{align*}
implies that
\[\AA_{\kappa}(\Gamma(q)) = \iota_{q^{-1}} \AA_{\kappa}\left(\Gamma_0\left(q^2\right) \cap \Gamma_1(q)\right),\]
where $\iota_{q^{-1}} : \AA_{\kappa}\left(\Gamma_0\left(q^2\right) \cap \Gamma_1(q)\right) \to \AA_{\kappa}(\Gamma(q))$ is the map $\iota_{q^{-1}} f(z) = f\left(q^{-1}z\right)$. As $\Gamma_0\left(q^2\right) \cap \Gamma_1(q)$ is a normal subgroup of $\Gamma_0\left(q^2\right)$ with quotient group isomorphic to $(\Z/q\Z)^{\times}$, we obtain the decomposition
\[\AA_{\kappa}\left(\Gamma(q)\right) = \bigoplus_{\substack{\chi \hspace{-.25cm}\pmod{q} \\ \chi(-1) = (-1)^{\kappa}}} \iota_{q^{-1}} \AA_{\kappa}\left(q^2,\chi\right),\]
thereby allowing us to choose an explicit basis $\BB_{\kappa}(\Gamma(q))$ of $\AA_{\kappa}(\Gamma(q))$ of the form
\begin{equation}\label{basisGamma(q)}
\bigsqcup_{\substack{\chi \hspace{-.25cm} \pmod{q} \\ \chi(-1) = (-1)^{\kappa}}} \bigsqcup_{\substack{q_1 q_2 = q^2 \\ q_1 \equiv 0 \hspace{-.25cm} \pmod{q_{\chi}}}} \bigsqcup_{f \in \BB_{\kappa}^{\ast}(q_1,\chi)} \bigsqcup_{d \mid q_2} \left\{\iota_{q^{-1}} f_d = \sum_{\ell \mid d} \xi_f(\ell,d)\iota_{q^{-1}}   \iota_{\ell,q_1,q} f \right\}.
\end{equation}
Once again, this allows us to make use of the pre-Kuznetsov and Kuznetsov formul\ae{} \eqref{pre-Kuznetsov} and \eqref{Kuznetsov} for $\BB_{\kappa}(\Gamma(q))$ and $\BB_0(\Gamma(q))$.

\section{Bounds for Fourier Coefficients of Newforms}

In the Kuznetsov formula \eqref{Kuznetsov}, the Fourier coefficients $|\rho_f(1)|^2$ and the normalisation factor $\xi_f$ both appear naturally. To remove these weights, we obtain lower bounds for $|\rho_f(1)|^2$ and $\xi_f$. For the former, such bounds are well-known, appearing in some generality in \cite[Equation (7.16)]{DFI}; nevertheless, we take this opportunity to correct some of the minor numerical errors in this proof, as well as greatly streamline the proof via the recent work of Li \cite{Li} on obtaining upper bounds for $L$-functions at the edge of the critical strip.

\begin{lemma}\label{xif(1)lemma}
For $f \in \BB_{\kappa}^{\ast}(q_1,\chi)$, we have that
\[\xi_f = \sum_{n \mid q_2^{\infty}} \frac{|\lambda_f(n)|^2}{n} \prod_{p \parallel q_2} \left(1 - \frac{\chi_{0(q_1)}(p)}{p^2}\right).\]
In particular, $\xi_f \gg 1$.
\end{lemma}

\begin{proof}
By multiplicativity,
\[\xi_f \defeq \sum_{d \mid q_2} |\xi_f(1,d)|^2 = \prod_{p^t \parallel q_2} \sum_{r = 0}^{t} |\xi_f(1,p^r)|^2.\]
We have that
\[\sum_{r = 0}^{t} |\xi_f(1,p^r)|^2 = \begin{dcases*}
1 & if $t = 0$,	\\
\frac{1}{1 - |A_f(p)|^2} & if $t = 1$,	\\
\frac{1}{(1 - \chi_{0(q_1)}(p) p^{-2}) (1 - |A_f(p)|^2)} & if $t \geq 2$.
\end{dcases*}\]
The result then follows from the fact that
\[\frac{1}{1 - |A_f(p)|^2} = \left(1 - \frac{\chi_{0(q_1)}(p)}{p^2}\right) \sum_{k = 0}^{\infty} \frac{|\lambda_f(p^k)|^2}{p^k}.\qedhere\]
\end{proof}

For $f \in \BB_{\kappa}(q,\chi)$, we define
\[\nu_f \defeq \Gamma\left(\frac{1 + \kappa}{2} + it_f\right) \Gamma\left(\frac{1 + \kappa}{2} - it_f\right) |\rho_f(1)|^2.\]
Note that
\[\Gamma\left(\frac{1 + \kappa}{2} + it\right) \Gamma\left(\frac{1 + \kappa}{2} - it\right) = \begin{dcases*}
\frac{\pi}{\cosh \pi t} & if $\kappa = 0$,	\\
\frac{\pi t}{\sinh \pi t} & if $\kappa = 1$.
\end{dcases*}\]

\begin{lemma}\label{L2idlemma}
Suppose that $f \in \BB_{\kappa}^{\ast}(q_1,\chi)$ for some $q_1 \mid q$. Then
\[\frac{\langle \iota_{1,q_1,q} f, \iota_{1,q_1,q} f \rangle_q}{\vol\left(\Gamma_0(q) \backslash \Hb\right)} = \nu_f \Res_{s = 1} \sum_{n = 1}^{\infty} \frac{|\lambda_f(n)|^2}{n^s}.\]
\end{lemma}

\begin{proof}
We let $\ell_1 = \ell_2 = 1$ in \eqref{F(s)} and take the residue at $s = 1$, yielding
\begin{multline*}
\frac{\langle \iota_{1,q_1,q} f, \iota_{1,q_1,q} f \rangle_q}{\vol\left(\Gamma_0(q) \backslash \Hb\right)} = |\rho_f(1)|^2 \Res_{s = 1} \sum_{n = 1}^{\infty} \frac{|\lambda_f(n)|^2}{n^s}	\\
\times \int_{0}^{\infty} \left(W_{\frac{\kappa}{2}, it_f}(y)^2 + \left|\frac{\Gamma\left(\frac{1 + \kappa}{2} + it_f\right)}{\Gamma\left(\frac{1 - \kappa}{2} + it_f\right)}\right|^2 W_{-\frac{\kappa}{2}, it_f}(y)^2\right) \, \frac{dy}{y},
\end{multline*}
since the residue of $E(z,s)$ at $s = 1$ is $1/\vol\left(\Gamma_0(q) \backslash \Hb\right)$. We have by \cite[7.611.4]{GR} that for $\kappa \in \C$ and $-1/2 < \Re(it) < 1/2$,
\[\int_{0}^{\infty} W_{\frac{\kappa}{2}, it}(y)^2 \, \frac{dy}{y} = \frac{\pi}{\sin 2\pi it} \frac{\psi\left(\frac{1 - \kappa}{2} + it\right) - \psi\left(\frac{1 - \kappa}{2} - it\right)}{\Gamma\left(\frac{1 - \kappa}{2} + it\right) \Gamma\left(\frac{1 - \kappa}{2} - it\right)},\]
where $\psi$ is the digamma function; note that a slightly erroneous version of this appears in \cite[Equation (19.6)]{DFI}. By the gamma and digamma reflection formul\ae{}, we find that
\begin{multline*}
\int_{0}^{\infty} \left(W_{\frac{\kappa}{2}, it_f}(y)^2 + \left|\frac{\Gamma\left(\frac{1 + \kappa}{2} + it_f\right)}{\Gamma\left(\frac{1 - \kappa}{2} + it_f\right)}\right|^2 W_{-\frac{\kappa}{2}, it_f}(y)^2\right) \, \frac{dy}{y}	\\
= \Gamma\left(\frac{1 + \kappa}{2} + it_f\right) \Gamma\left(\frac{1 + \kappa}{2} - it_f\right)
\end{multline*}
assuming that $t_f \in [0,\infty)$ if $\kappa = 1$ and $t_f \in [0,\infty)$ or $it_f \in (0,1/2)$ if $\kappa = 0$. 
\end{proof}

\begin{corollary}
Suppose that $f \in \BB_{\kappa}^{\ast}\left(q_1,\chi\right)$ for some $q_1 \mid q$. Then
\begin{equation}\label{rho1lowerbound}
\nu_f \gg_{\e} \frac{\langle \iota_{1,q_1,q} f, \iota_{1,q_1,q} f \rangle_q}{\vol\left(\Gamma_0(q) \backslash \Hb\right)} \left(q \left(3 + t_f^2\right)\right)^{-\e}.
\end{equation}
\end{corollary}

\begin{proof}
It is known that
\[\sum_{n = 1}^{\infty} \frac{|\lambda_f(n)|^2}{n^s} = \frac{\zeta(s) L(s, \ad f)}{\zeta(2s)} \prod_{p \mid q} P_{f,p}(p^{-s}),\]
where for each prime $p$ dividing $q$, $P_{f,p}(z)$ is a rational function satisfying $p^{-\e} \ll_{\e} P_{f,p}(p^{-1}) \leq 1$. The work of Li \cite[Theorem 2]{Li} then shows that
\[L(1, \ad f) \ll \exp\left(C \frac{\log \left(q \left(3 + t_f^2\right)\right)}{\log \log \left(q \left(3 + t_f^2\right)\right)}\right)\]
for some absolute constant $C > 0$, thereby yielding the result.
\end{proof}

\section{Bounds for Sums of Kloosterman Sums}

We denote by
\[S(m,n;c) \defeq \sum_{d \in (\Z/c\Z)^{\times}} e\left(\frac{md + n\overline{d}}{c}\right)\]
the usual Kloosterman sum with trivial character, for which the Weil bound holds:
\begin{equation}\label{Weilbound}
|S(m,n;c)|\leq \tau(c) \sqrt{(m,n,c) c}.
\end{equation}
We also require bounds for Kloosterman sums with nontrivial character. For $c \equiv 0 \pmod{q}$, $m,n \geq 1$, and $(a,q) = 1$, we have that
\begin{multline*}
\sum_{\substack{\chi \hspace{-.25cm} \pmod{q} \\ \chi(-1) = (-1)^{\kappa}}} \overline{\chi}(a) S_{\chi}(m,n;c)	\\
= \frac{1}{2} \sum_{d \in (\Z/c\Z)^{\times}} \sum_{\chi \hspace{-.25cm} \pmod{q}} \overline{\chi}(a) \left(\chi(d) + (-1)^{\kappa} \chi(-d)\right) e\left(\frac{md + n\overline{d}}{c}\right).
\end{multline*}
We break this up into two sums. In the second sum, we can replace $d$ with $-d$ and $\chi$ with $\overline{\chi}$ and use character orthogonality to see that
\begin{equation}\label{sumKloosterman}
\sum_{\substack{\chi \hspace{-.25cm} \pmod{q} \\ \chi(-1) = (-1)^{\kappa}}} \overline{\chi}(a) S_{\chi}(m,n;c) =
\begin{dcases*}
\varphi(q) \Re\left(S_{a(q)}(m,n;c)\right) & if $\kappa = 0$,	\\
i \varphi(q) \Im\left(S_{a(q)}(m,n;c)\right) & if $\kappa = 1$,
\end{dcases*}
\end{equation}
where we set
\[S_{a(q)}(m,n;c) \defeq \sum_{\substack{d \in (\Z/c\Z)^{\times} \\ d \equiv a \hspace{-.25cm} \pmod{q}}} e\left(\frac{md + n\overline{d}}{c}\right).\]

If $c = c_1 c_2$ with $(c_1,c_2) = 1$ and $c_1 c_2 \equiv 0 \pmod{q}$, then we let $d = c_2 \overline{c_2} d_1 + c_1 \overline{c_1} d_2$, where $d_1 \in (\Z/c_1 \Z)^{\times}$, $d_2 \in (\Z/c_2 \Z)^{\times}$, and $c_2 \overline{c_2} \equiv 1 \pmod{c_1}$, $c_1 \overline{c_1} \equiv 1 \pmod{c_2}$. By the Chinese remainder theorem,
\[S_{a(q)}(m,n;c) = S_{a((q,c_1))}\left(m\overline{c_2}, n\overline{c_2}; c_1\right) S_{a((q,c_2))}\left(m\overline{c_1}, n\overline{c_1}; c_2\right).\]
To bound $S_{a(q)}(m,n;c)$, it therefore suffices to find bounds for $S_{a(p^{\alpha})}\left(m,n;p^{\beta}\right)$ for any prime $p$ and any $\beta \geq \alpha \geq 1$. The trivial bound is merely
\begin{equation}\label{trivKloostbound}
\left|S_{a(p^{\alpha})}\left(m,n;p^{\beta}\right)\right| \leq p^{\beta - \alpha}.
\end{equation}
Somewhat surprisingly, this is sufficient for our needs. Indeed, we cannot do better than this when $\beta = \alpha$, and in our applications, this will be the dominant contribution.

We also require bounds for $S_{\chi}(m,n;c)$. Unfortunately, it is not necessarily the case that this is bounded by $\tau(c) \sqrt{(m,n,c) c}$, which can be observed numerically at \cite{LMFDB}; see also \cite[Example 9.9]{KL}.

\begin{lemma}\label{weakWeillemma}
Let $p$ be an odd prime, let $\chi_{p^{\gamma}}$ be a Dirichlet character of conductor $p^{\gamma}$, and suppose that $(mn,p) = 1$. Then for $\beta \geq \gamma \geq 0$, we have that
\[\left|S_{\chi_{p^{\gamma}}}(m,n;p^{\beta})\right| \leq 2 p^{\beta/2}\]
unless $\beta = \gamma \geq 3$, in which case we only have that
\[\left|S_{\chi_{p^{\gamma}}}(m,n;p^{\beta})\right| \leq 2 p^{\lfloor \frac{3\beta + 1}{4} \rfloor}.\]

Similarly, let $\chi_{2^{\gamma}}$ be a Dirichlet character of conductor $2^{\gamma}$, and suppose that $(mn,2) = 1$. Then for $\beta \geq \gamma \geq 0$, we have that
\[\left|S_{\chi_{2^{\gamma}}}(m,n;2^{\beta})\right| \leq 8 \cdot 2^{\beta/2}\]
unless $\gamma + 1 \geq \beta \geq 3$, in which case we only have that
\[\left|S_{\chi_{2^{\gamma}}}(m,n;2^{\beta})\right| \leq 4 \cdot 2^{\lfloor \frac{3\beta + 1}{4} \rfloor}.\]
\end{lemma}

\begin{proof}
This follows from \cite[Propositions 9.4, 9.7, 9.8, and Lemmata 9.6]{KL}.
\end{proof}

\begin{lemma}
When $(m,n) = 1$, we have that
\begin{align}
\sum_{\substack{c \leq 4\pi \sqrt{mn} \\ c \equiv 0 \hspace{-.25cm} \pmod{q}}} \frac{\left|S_{a(q)}(m,n;c)\right|}{c^{3/2}} & \ll \frac{(\log (mn + 1))^2}{q^{3/2}} \prod_{p \mid q} \frac{1}{1 - p^{-1/2}},	\label{Gamma1(q)cleq}\\
\sum_{\substack{c \leq 4\pi \sqrt{mn} \\ c \equiv 0 \hspace{-.25cm} \pmod{q^2}}} \frac{\left|S_{a(q)}(m,n;c)\right|}{c^{3/2}} & \ll \frac{(\log (mn + 1))^2}{q^2} \prod_{p \mid q} \frac{1}{1 - p^{-1/2}}.	\label{Gamma(q)cleq}
\end{align}
If we additionally assume that $(mn,q) = 1$, then given a Dirichlet character $\chi$ modulo $q$, we have that
\begin{equation}\label{Gamma0(q)cleq}
\sum_{\substack{c \leq 4\pi \sqrt{mn} \\ c \equiv 0 \hspace{-.25cm} \pmod{q}}} \frac{\left|S_{\chi}(m,n;c)\right|}{c^{3/2}} \ll (\log (mn + 1))^2 \frac{2^{\omega(q)} \dot{Q}}{\varphi(q)}.
\end{equation}
\end{lemma}

\begin{proof}
We write $q = p_1^{\alpha_1} \cdots p_{\ell}^{\alpha_{\ell}}$, so that the left-hand side of \eqref{Gamma1(q)cleq} is
\begin{multline*}
\sum_{\beta_1 = \alpha_1}^{\infty} \cdots \sum_{\beta_{\ell} = \alpha_{\ell}}^{\infty} \frac{1}{\left(p_1^{\beta_1} \cdots p_{\ell}^{\beta_{\ell}}\right)^{3/2}} \sum_{\substack{c \leq 4\pi \sqrt{mn} p_1^{-\beta_1} \cdots p_{\ell}^{-\beta_{\ell}} \\ (c,q) = 1}} \frac{1}{c^{3/2}}	\\
\times \left|S\left(m \overline{p_1^{\beta_1} \cdots p_{\ell}^{\beta_{\ell}}}, n \overline{p_1^{\beta_1} \cdots p_{\ell}^{\beta_{\ell}}}; c\right)\right| \left|S_{a(q)}\left(m \overline{c}, n \overline{c}; p_1^{\beta_1} \cdots p_{\ell}^{\beta_{\ell}}\right)\right|.
\end{multline*}
Using the Weil bound \eqref{Weilbound} for the first Kloosterman sum and the trivial bound \eqref{trivKloostbound} for the second, we find that this is bounded by
\[\frac{1}{q} \sum_{\beta_1 = \alpha_1}^{\infty} \cdots \sum_{\beta_{\ell} = \alpha_{\ell}}^{\infty} \frac{1}{\sqrt{p_1^{\beta_1} \cdots p_{\ell}^{\beta_{\ell}}}} \sum_{\substack{c \leq 4\pi \sqrt{mn} \\ (c,q) = 1}} \frac{\tau(c) \sqrt{(m,n,c)}}{c}.\]
If $(m,n) = 1$, the inner sum is bounded by a constant multiple of $(\log(mn + 1))^2$, and so the sum is bounded by a constant multiple of
\[\frac{(\log (mn + 1))^2}{q} \sum_{\beta_1 = \alpha_1}^{\infty} \cdots \sum_{\beta_{\ell} = \alpha_{\ell}}^{\infty} \frac{1}{\sqrt{p_1^{\beta_1} \cdots p_{\ell}^{\beta_{\ell}}}} ,\]
which yields \eqref{Gamma1(q)cleq} upon evaluating these geometric series. \eqref{Gamma(q)cleq} follows similarly. Finally, \eqref{Gamma0(q)cleq} follows via the same method but using \hyperref[weakWeillemma]{Lemma \ref*{weakWeillemma}} to bound the Kloosterman sums, yielding the bound
\[8 \cdot 2^{\omega(q)} \dot{Q} \sum_{\beta_1 = \alpha_1}^{\infty} \cdots \sum_{\beta_{\ell} = \alpha_{\ell}}^{\infty} \frac{1}{p_1^{\beta_1} \cdots p_{\ell}^{\beta_{\ell}}} \sum_{\substack{c \leq 4\pi \sqrt{mn} \\ (c,q) = 1}} \frac{\tau(c)}{c}\]
for the left-hand side of \eqref{Gamma0(q)cleq}, from which the result easily follows.
\end{proof}

\begin{lemma}
When $(m,n) = 1$, we have that
\begin{align}
\sum_{\substack{c > 4\pi \sqrt{mn} \\ c \equiv 0 \hspace{-.25cm} \pmod{q}}} \frac{\left|S_{a(q)}(m,n;c)\right|}{c^2} \left(1 + \log \frac{c}{4\pi\sqrt{mn}}\right) & \ll \frac{(\log (mn + 1))^2}{(mn)^{1/4}} \frac{1}{q^{3/2}} \prod_{p \mid q} \frac{1}{1 - p^{-1/2}},	\label{Gamma1(q)cgeq}\\
\sum_{\substack{c > 4\pi \sqrt{mn} \\ c \equiv 0 \hspace{-.25cm} \pmod{q^2}}} \frac{\left|S_{a(q)}(m,n;c)\right|}{c^2} \left(1 + \log \frac{c}{4\pi\sqrt{mn}}\right) & \ll \frac{(\log (mn + 1))^2}{(mn)^{1/4}} \frac{1}{q^2} \prod_{p \mid q} \frac{1}{1 - p^{-1/2}}.	\label{Gamma(q)cgeq}
\end{align}
If we additionally assume that $(mn,q) = 1$, then given a Dirichlet character $\chi$ modulo $q$, we have that
\begin{equation}\label{Gamma0(q)cgeq}
\sum_{\substack{c > 4\pi \sqrt{mn} \\ c \equiv 0 \hspace{-.25cm} \pmod{q}}} \frac{\left|S_{\chi}(m,n;c)\right|}{c^2} \left(1 + \log \frac{c}{4\pi\sqrt{mn}}\right) \ll \frac{(\log (mn + 1))^2}{(mn)^{1/4}} \frac{2^{\omega(q)} \dot{Q}}{\varphi(q)}.
\end{equation}
\end{lemma}

\begin{proof}
As before, with $q = p_1^{\alpha_1} \cdots p_{\ell}^{\alpha_{\ell}}$, the left-hand side of \eqref{Gamma1(q)cgeq} is bounded by
\[\frac{1}{q} \sum_{\beta_1 = \alpha_1}^{\infty} \cdots \sum_{\beta_{\ell} = \alpha_{\ell}}^{\infty} \frac{1}{p_1^{\beta_1} \cdots p_{\ell}^{\beta_{\ell}}} \sum_{\substack{c > 4\pi \sqrt{mn} p_1^{-\beta_1} \cdots p_{\ell}^{-\beta_{\ell}} \\ (c,q) = 1}} \frac{\tau(c) \sqrt{(m,n,c)} \log c}{c^{3/2}}.\]
If $(m,n) = 1$, then the inner sum is bounded by a constant multiple of
\[\frac{(\log (mn + 1))^2}{(mn)^{1/4}} \sqrt{p_1^{\beta_1} \cdots p_{\ell}^{\beta_{\ell}}}.\]
It follows that the sum is bounded by a constant multiple of
\[\frac{(\log (mn + 1))^2}{(mn)^{1/4}} \frac{1}{q} \sum_{\beta_1 = \alpha_1}^{\infty} \cdots \sum_{\beta_{\ell} = \alpha_{\ell}}^{\infty} \frac{1}{\sqrt{p_1^{\beta_1} \cdots p_{\ell}^{\beta_{\ell}}}},\]
which gives \eqref{Gamma1(q)cgeq}. The proof of \eqref{Gamma(q)cgeq} is analogous, while \eqref{Gamma0(q)cgeq} again follows upon using \hyperref[weakWeillemma]{Lemma \ref*{weakWeillemma}} to bound the Kloosterman sums.
\end{proof}

\begin{lemma}[{Cf.~\cite[Equation (16.50)]{IK}}]
For all $1/2 < \sigma < 1$,
\begin{align}
\sum_{\substack{c = 1 \\ c \equiv 0 \hspace{-.25cm} \pmod{q}}}^{\infty} \frac{\left|S_{a(q)}(m,n;c)\right|}{c^{1 + \sigma}} & \leq \frac{18 \tau((m,n))}{(2\sigma - 1)^2} \frac{1}{q^{1 + \sigma}} \prod_{p \mid q} \frac{1}{1 - p^{-\sigma}},	\label{Gamma1(q)csigma}\\
\sum_{\substack{c = 1 \\ c \equiv 0 \hspace{-.25cm} \pmod{q^2}}}^{\infty} \frac{\left|S_{a(q)}(m,n;c)\right|}{c^{1 + \sigma}} & \leq \frac{18 \tau((m,n))}{(2\sigma - 1)^2} \frac{1}{q^{1 + 2\sigma}} \prod_{p \mid q} \frac{1}{1 - p^{-\sigma}}.	\label{Gamma(q)csigma}
\end{align}
If we additionally assume that $(m,n) = (mn,q) = 1$, then given a Dirichlet character $\chi$ modulo $q$, we have that
\begin{equation}\label{Gamma0(q)csigma}
\sum_{\substack{c = 1 \\ c \equiv 0 \hspace{-.25cm} \pmod{q}}}^{\infty} \frac{\left|S_{\chi}(m,n;c)\right|}{c^{1 + \sigma}} \leq \frac{72}{(2\sigma - 1)^2} \frac{2^{\omega(q)} \dot{Q}}{\varphi(q) q^{\sigma - 1/2}}.
\end{equation}
\end{lemma}

\begin{proof}
Once again writing $q = p_1^{\alpha_1} \cdots p_{\ell}^{\alpha_{\ell}}$ and bounding the Kloosterman sums, we have that
\begin{align*}
\sum_{\substack{c = 1 \\ c \equiv 0 \hspace{-.25cm} \pmod{q}}}^{\infty} \frac{\left|S_{a(q)}(m,n;c)\right|}{c^{1 + \sigma}} & \leq \sum_{\substack{c = 1 \\ (c,q) = 1}}^{\infty} \frac{\tau(c) \sqrt{(m,n,c)}}{c^{1/2 + \sigma}} \frac{1}{q} \sum_{\beta_1 = \alpha_1}^{\infty} \cdots \sum_{\beta_{\ell} = \alpha_{\ell}}^{\infty} \frac{1}{\left(p_1^{\beta_1} \cdots p_{\ell}^{\beta_{\ell}}\right)^{\sigma}}	\\
& = \sum_{\substack{c = 1 \\ (c,q) = 1}}^{\infty} \frac{\tau(c) \sqrt{(m,n,c)}}{c^{1/2 + \sigma}} \frac{1}{q^{1 + \sigma}} \prod_{p \mid q} \frac{1}{1 - p^{-\sigma}}	\\
& \leq \zeta\left(\sigma + \frac{1}{2}\right)^2 \sum_{d \mid (m,n)} \frac{\tau(d)}{d^{\sigma}} \frac{1}{q^{1 + \sigma}} \prod_{p \mid q} \frac{1}{1 - p^{-\sigma}}	\\
& \leq \frac{18 \tau((m,n))}{(2\sigma - 1)^2} \frac{1}{q^{1 + \sigma}} \prod_{p \mid q} \frac{1}{1 - p^{-\sigma}}.
\end{align*}
This proves \eqref{Gamma1(q)csigma}. The inequality \eqref{Gamma(q)csigma} follows by a similar argument, as does \eqref{Gamma0(q)csigma} once the Kloosterman sums are bounded via \hyperref[weakWeillemma]{Lemma \ref*{weakWeillemma}}.
\end{proof}

\section{Bounds for Test Functions}

We require bounds for the test function that we will obtain by multiplying the pre-Kuznetsov formula \eqref{pre-Kuznetsov} by a function dependent on $r$ and then integrating both sides over $r \in [0,T]$.

\begin{lemma}\label{hkTlemma}
For $T \geq 1$, let
\begin{align*}
h_{\kappa,T}(t) & \defeq \frac{\pi^2}{\Gamma\left(\frac{1 + \kappa}{2} + it\right) \Gamma\left(\frac{1 + \kappa}{2} - it\right)} \int_{0}^{T} \frac{r \left|\Gamma\left(1 - \frac{\kappa}{2} + ir\right)\right|^{-2}}{\cosh \pi (r - t) \cosh \pi (r + t)} \, dr	\\
& = \begin{dcases*}
\cosh \pi t \int_{0}^{T} \frac{\sinh \pi r}{\cosh \pi (r - t) \cosh \pi (r + t)} \, dr & if $\kappa = 0$,	\\
\frac{\sinh \pi t}{t} \int_{0}^{T} \frac{r \cosh \pi r}{\cosh \pi (r - t) \cosh \pi (r + t)} \, dr & if $\kappa = 1$.
\end{dcases*}
\end{align*}
Then $h_{\kappa,T}(t)$ is positive for all $t \in \R$ and additionally, should $\kappa$ be equal to $0$, for $it \in (-1/2,1/2)$. Furthermore, $h_{\kappa,T}(t) \gg 1$ for $t \in [0,T]$.
\end{lemma}

\begin{proof}
Using the fact that
\[\cosh \pi (r - t) \cosh \pi (r + t) = \cosh^2 \pi t + \sinh^2 \pi r = \sinh^2 \pi t + \cosh^2 \pi r,\]
it is clear that $h_{\kappa,T}(t)$ is positive for all $t \in \R$ and additionally, should $\kappa$ be equal to $0$, if $it \in (-1/2,1/2)$.

For $\kappa = 0$, we have that
\begin{align*}
h_{0,T}(t) & = \frac{\cosh \pi t}{\pi} \int_{1}^{\cosh \pi T} \frac{1}{x^2 + \sinh^2 \pi t} \, dx	\\
& = \frac{\coth \pi t}{\pi} \arctan \frac{\sinh \pi t \left(\cosh \pi T - 1\right)}{\sinh^2 \pi t + \cosh\pi T},
\end{align*}
where the second line follows from the arctangent subtraction formula. The first expression shows that $h_{0,T}(t) \gg 1$ when $t$ is small, while when $t$ is large, the argument of $\arctan$ is essentially
\[\frac{e^{\pi(T + t)} - e^{\pi t}}{e^{2\pi t} + e^{\pi T}},\]
and this is bounded from below provided that $t \leq T$, so that again $h_{0,T}(t) \gg 1$.

For $\kappa = 1$, we can similarly show via integration by parts that
\begin{align*}
h_{1,T}(t) & = \frac{\sinh \pi t}{\pi^2 t} \int_{0}^{\sinh \pi T} \frac{\arsinh x}{x^2 + \cosh^2 \pi t} \, dx	\\
& = \frac{\tanh \pi t}{\pi^2 t} \int_{0}^{\sinh \pi T} \frac{ \arctan \frac{\sinh \pi T}{\cosh \pi t} - \arctan \frac{x}{\cosh \pi t}}{\sqrt{x^2 + 1}} \, dx.
\end{align*}
The first expression shows that $h_{1,T}(t) \gg 1$ when $t$ is small, while when $t$ is large, we break up the second expression into two integrals: one from $0$ to $\sinh \frac{\pi t}{2}$ and one from $\sinh \frac{\pi t}{2}$ to $\sinh \pi T$. Trivially bounding the numerator in each integral, we find that
\begin{align*}
h_{1,T}(t) & \geq \frac{\tanh \pi t}{2\pi} \left(\arctan \frac{\sinh \pi T}{\cosh \pi t} - \arctan \frac{\sinh \frac{\pi t}{2}}{\cosh \pi t}\right)	\\
& = \frac{\tanh \pi t}{2\pi} \arctan \frac{\cosh \pi t \left(\sinh \pi T - \sinh \frac{\pi t}{2}\right)}{\cosh^2 \pi t + \sinh \pi T \sinh \frac{\pi t}{2}}.
\end{align*}
The argument of $\arctan$ is essentially
\[\frac{e^{\pi(T + t)} - e^{3\pi t/2}}{e^{2\pi t} + e^{\pi(T + t/2)}},\]
and this is bounded from below provided that $t \leq T$, while $\tanh \pi t$ is bounded from below provided that $t$ is larger than some fixed constant. It follows again that $h_{1,T}(t) \gg 1$.
\end{proof}

We also require the following bound, which arises from the Kloosterman term in the pre-Kuznetsov formula \eqref{pre-Kuznetsov}.

\begin{lemma}
For $\kappa \in \{0,1\}$ and $T > 0$, we have the bound
\begin{equation}\label{uniformIkint}
\int_{0}^{T} r I_{\kappa}(a,r) \, dr \ll \begin{dcases*}
\sqrt{a} & if $a \geq 1$,	\\
a\left(1 + \log \frac{1}{a}\right) & if $0 < a < 1$
\end{dcases*}
\end{equation}
uniformly in $T$.
\end{lemma}

\begin{proof}
From \cite[Equation (5.13)]{Kuz}, we have that
\[\int_{0}^{T} r I_0(a,r) \, dr = a \int_{0}^{\infty} \frac{\tanh \xi}{\xi} (1 - \cos 2T \xi) \sin(a \cosh \xi) \, d\xi.\]
Similarly, using the fact that
\[K_{2ir}(\zeta) = \int_{0}^{\infty} e^{-\zeta \cosh \xi} \cos 2r \xi \, d\xi\]
for $r \in \R$ and $\Re(\zeta) > 0$ from \cite[8.432.1]{GR}, we have that
\[\int_{0}^{T} r I_1(a,r) \, dr = -2a \int_{0}^{\infty} \int_{0}^{T} r \cos 2r \xi \, dr \int_{-i}^{i} e^{-\zeta a \cosh \xi} \, d\zeta \, d\xi.\]
Evaluating each of the inner integrals and then integrating by parts, we find that
\begin{multline*}
\int_{0}^{T} r I_1(a,r) \, dr = i a \int_{0}^{\infty} \frac{\tanh \xi}{\xi} (1 - \cos 2T \xi) \cos(a \cosh \xi) \, d\xi	\\
- i \int_{0}^{\infty} \frac{\tanh \xi}{\xi} (1 - \cos 2T \xi) \frac{\sin(a \cosh \xi)}{\cosh \xi} \, d\xi.
\end{multline*}
From here, one can show via stationary phase on subintervals of $(0,\infty)$ that $\int_{0}^{T} r I_0(a,r) \, dr$ and the first term in the above expression for $\int_{0}^{T} r I_1(a,r) \, dr$ both are bounded by a constant multiple of
\[\begin{dcases*}
\sqrt{a} & if $a \geq 1$,	\\
a\left(1 + \log \frac{1}{a}\right) & if $0 < a < 1$;
\end{dcases*}\]
see \cite[Equation (5.14)]{Kuz}. The second term in the expression for $\int_{0}^{T} r I_1(a,r) \, dr$ is uniformly bounded for $a \geq 1$, so we need only consider when $0 < a < 1$. In this case, the fact that $|\sin x| \leq \min\{1,|x|\}$ for $x \in \R$ implies that this is bounded by
\[2a \int_{0}^{\log \frac{1}{a}} \frac{\tanh \xi}{\xi} \, d\xi + 2 \int_{\log \frac{1}{a}}^{\infty} \frac{\tanh \xi}{\xi} \frac{1}{\cosh \xi} \, d\xi	\\
\ll a \left(1 + \log \frac{1}{a}\right).\qedhere\]
\end{proof}

\section{Sarnak's Density Theorem for Exceptional Hecke Eigenvalues}

We are now in a position to prove \hyperref[Sarnakthm]{Theorem \ref*{Sarnakthm}}.

\begin{proof}[Proof of {\eqref{SarnakGamma1}}]
By Rankin's trick,
\begin{multline*}
\#\left\{f \in \BB_{\kappa}\left(\Gamma_1(q)\right) : t_f \in [0,T], \ \left|\lambda_f(p)\right| \geq \alpha_p \text{ for all $p \in \PP$}\right\}	\\
\leq \prod_{p \in \PP} \alpha_p^{-2\ell_p} \sum_{\substack{f \in \BB_{\kappa}\left(\Gamma_1(q)\right) \\ t_f \in [0,T]}} \prod_{p \in \PP} \left|\lambda_f(p)\right|^{2\ell_p}
\end{multline*}
for any nonnegative integers $\ell_p$ to be chosen. Using the explicit basis \eqref{basisGamma1(q)} of $\AA_{\kappa}\left(\Gamma_1(q)\right)$ together with the lower bound \eqref{rho1lowerbound} for $\nu_f$,
\begin{align*}
\hspace{2cm} & \hspace{-2cm} \sum_{\substack{f \in \BB_{\kappa}\left(\Gamma_1(q)\right) \\ t_f \in [0,T]}} \prod_{p \in \PP} \left|\lambda_f(p)\right|^{2\ell_p}	\\
& = \sum_{\substack{\chi \hspace{-.25cm} \pmod{q} \\ \chi(-1) = (-1)^{\kappa}}} \sum_{\substack{q_1 q_2 = q \\ q_1 \equiv 0 \hspace{-.25cm} \pmod{q_{\chi}}}} \sum_{\substack{f \in \BB_{\kappa}^{\ast}\left(q_1,\chi\right) \\ t_f \in [0,T]}} \tau(q_2) \prod_{p \in \PP} \left|\lambda_f(p)\right|^{2\ell_p}	\\
& \ll_{\e} q^{1 + \e} T^{\e} \sum_{\substack{\chi \hspace{-.25cm} \pmod{q} \\ \chi(-1) = (-1)^{\kappa}}} \sum_{\substack{q_1 q_2 = q \\ q_1 \equiv 0 \hspace{-.25cm} \pmod{q_{\chi}}}} \sum_{\substack{f \in \BB_{\kappa}^{\ast}\left(q_1,\chi\right) \\ t_f \in [0,T]}} \xi_f \nu_f \prod_{p \in \PP} \left|\lambda_f(p)\right|^{2\ell_p}.
\end{align*}

We take $m = 1$ and $n = \prod_{p \in \PP} p^{2j_p}$ in the pre-Kuznetsov formula \eqref{pre-Kuznetsov}, multiply both sides by $\prod_{p \in \PP} \alpha_{2j_p,2\ell_p} \overline{\chi}(p)^{j_p}$, and sum over all $0 \leq j_p \leq \ell_p$, over all $p \in \PP$, and over all Dirichlet characters $\chi$ modulo $q$ satisfying $\chi(-1) = (-1)^{\kappa}$. We then multiply both sides by $\pi^2 r \left|\Gamma\left(1 - \frac{\kappa}{2} + ir\right)\right|^{-2}$ and integrate both sides with respect to $r$ from $0$ to $T$. 

On the spectral side, \eqref{Eisensteinmult}, \eqref{EisensteinHecke2ell}, and \hyperref[hkTlemma]{Lemma \ref*{hkTlemma}} allow us to use positivity to discard the contribution from the continuous spectrum, while we may discard the contribution of the discrete spectrum with $t \notin [0,T]$ via \eqref{cuspmult}, \eqref{cuspHecke2ell}, and \hyperref[hkTlemma]{Lemma \ref*{hkTlemma}}, so that the spectral side is bounded from below by a constant multiple of
\[\sum_{\substack{\chi \hspace{-.25cm} \pmod{q} \\ \chi(-1) = (-1)^{\kappa}}} \sum_{\substack{q_1 q_2 = q \\ q_1 \equiv 0 \hspace{-.25cm} \pmod{q_{\chi}}}} \sum_{\substack{f \in \BB_{\kappa}^{\ast}\left(q_1,\chi\right) \\ t_f \in [0,T]}} \xi_f \nu_f \prod_{p \in \PP} \left|\lambda_f(p)\right|^{2\ell_p}.\]
On the geometric side, we only pick up the delta term when $j_p = 0$ for all $p \in \PP$, in which case the term is bounded by a constant multiple of $q T^2 \prod_{p \in \PP} \alpha_{0,2\ell_p}$. For $\kappa = 0$, we use \eqref{sumKloosterman} to write the Kloosterman term in the form
\begin{multline*}
\frac{\varphi(q)}{\pi} \sum_{\substack{j_p = 0 \\ p \in \PP}}^{\ell_p} \prod_{p \in \PP} \alpha_{2j_p,2\ell_p} \sum_{\substack{c = 1 \\ c \equiv 0 \hspace{-.25cm} \pmod{q}}}^{\infty} \frac{\Re\left(S_{\prod_{p \in \PP} p^{j_p} (q)}\left(1,\prod_{p \in \PP} p^{2j_p};c\right)\right)}{c}	\\
\times \int_{0}^{T} r I_0\left(\frac{4\pi \prod_{p \in \PP} p^{j_p}}{c},r\right) \, dr.
\end{multline*}
For $\kappa = 1$, the Kloosterman term is the same except with $i\Im$ in place of $\Re$ and $I_1$ in place of $I_0$. In either case, we bound the integral via \eqref{uniformIkint}, which allows us to use \eqref{Gamma1(q)cleq} and \eqref{Gamma1(q)cgeq} to bound the summation over $c$, so that the Kloosterman term is bounded by a constant multiple of
\[\frac{1}{\sqrt{q}} \prod_{p' \mid q} \frac{1}{1 - {p'}^{-1/2}} \sum_{\substack{j_p = 0 \\ p \in \PP}}^{\ell_p} \prod_{p \in \PP} \alpha_{2j_p,2\ell_p} p^{j_p/2} \left(\log \left(\prod_{p \in \PP} p^{2j_p} + 1\right)\right)^2.\]
We bound the summation over $j_p$ and over $p \in \PP$ via \eqref{sumalphabound}, thereby obtaining
\begin{multline*}
\#\left\{f \in \BB_{\kappa}\left(\Gamma_1(q)\right) : t_f \in [0,T], \ \left|\lambda_f(p)\right| \geq \alpha_p \text{ for all $p \in \PP$}\right\}	\\
\ll_{\e} q^{1 + \e} T^{\e} \prod_{p \in \PP} \left(\frac{\alpha_p}{2}\right)^{-2\ell_p} \left(qT^2 + \frac{\prod_{p \in \PP} p^{\ell_p/2} \left(\log \prod_{p \in \PP} p^{\ell_p/2}\right)^2}{\sqrt{q}} \prod_{p' \mid q} \frac{1}{1 - {p'}^{-1/2}}\right).
\end{multline*}
It remains to take
\[\ell_p = \left\lfloor \frac{\mu_p \log \left(\vol\left(\Gamma_1(q) \backslash \Hb\right)^{3/2} T^4\right)}{\log p} \right\rfloor.\qedhere\]
\end{proof}

\begin{proof}[Proof of {\eqref{SarnakGamma}}]
We use \eqref{basisGamma(q)}, \eqref{Gamma(q)cleq}, and \eqref{Gamma(q)cgeq} in place of \eqref{basisGamma1(q)}, \eqref{Gamma1(q)cleq}, and \eqref{Gamma1(q)cgeq}, thereby finding that
\begin{multline*}
\#\left\{f \in \BB_{\kappa}\left(\Gamma(q)\right) : t_f \in [0,T], \ \left|\lambda_f(p)\right| \geq \alpha_p \text{ for all $p \in \PP$}\right\}	\\
\leq \prod_{p \in \PP} \alpha_p^{-2\ell_p} \sum_{\substack{f \in \BB_{\kappa}\left(\Gamma(q)\right) \\ t_f \in [0,T]}} \prod_{p \in \PP} \left|\lambda_f(p)\right|^{2\ell_p},
\end{multline*}
with
\begin{align*}
& \sum_{\substack{f \in \BB_{\kappa}\left(\Gamma(q)\right) \\ t_f \in [0,T]}} \prod_{p \in \PP} \left|\lambda_f(p)\right|^{2\ell_p}	\\
& = \sum_{\substack{\chi \hspace{-.25cm} \pmod{q} \\ \chi(-1) = (-1)^{\kappa}}} \sum_{\substack{q_1 q_2 = q^2 \\ q_1 \equiv 0 \hspace{-.25cm} \pmod{q_{\chi}}}} \sum_{\substack{f \in \BB_{\kappa}^{\ast}\left(q_1,\chi\right) \\ t_f \in [0,T]}} \tau(q_2) \prod_{p \in \PP} \left|\lambda_f(p)\right|^{2\ell_p}	\\
& \ll_{\e} q^{2 + \e} T^{\e} \prod_{p \in \PP} 2^{2\ell_p} \left(qT^2 + \frac{\prod_{p \in \PP} p^{\ell_p/2} \left(\log \prod_{p \in \PP} p^{\ell_p/2}\right)^2}{q} \prod_{p' \mid q} \frac{1}{1 - {p'}^{-1/2}}\right).
\end{align*}
Taking
\[\ell_p = \left\lfloor \frac{\mu_p \log \left(\vol\left(\Gamma(q) \backslash \Hb\right)^{4/3} T^4\right)}{\log p} \right\rfloor\]
completes the proof.
\end{proof}

\begin{proof}[Proof of {\eqref{SarnakGamma0}}]
Using \eqref{basisqchi}, \eqref{Gamma0(q)cleq}, and \eqref{Gamma0(q)cgeq} in place of \eqref{basisGamma1(q)}, \eqref{Gamma1(q)cleq}, and \eqref{Gamma1(q)cgeq},
\begin{multline*}
\#\left\{f \in \BB_{\kappa}(q,\chi) : t_f \in [0,T], \ \left|\lambda_f(p)\right| \geq \alpha_p \text{ for all $p \in \PP$}\right\}	\\
\ll_{\e} q^{1 + \e} T^{\e} \prod_{p \in \PP} \left(\frac{\alpha_p}{2}\right)^{-2\ell_p} \left(T^2 + \prod_{p \in \PP} p^{\ell_p/2} \left(\log \prod_{p \in \PP} p^{\ell_p/2}\right)^2 \frac{2^{\omega(q)} \dot{Q}}{\varphi(q)}\right).
\end{multline*}
Upon taking
\[\ell_p = \left\lfloor \frac{\mu_p \log \left(\vol\left(\Gamma_0(q) \backslash \Hb\right)^2 T^4 \dot{Q}^{-2} \right)}{\log p} \right\rfloor,\]
we conclude that
\begin{multline}\label{SarnakGamma0(q)dotq}
\#\left\{f \in \BB_{\kappa}(q,\chi) : t_f \in [0,T], \ \left|\lambda_f(p)\right| \geq \alpha_p \text{ for all $p \in \PP$}\right\}	\\
\ll_{\e} \left(\vol(\Gamma_0(q) \backslash \Hb) T^2\right)^{1 - 4 \sum_{p \in \PP} \mu_p \frac{\log \alpha_p/2}{\log p} + \e} \dot{Q}^{4 \sum_{p \in \PP} \mu_p \frac{\log \alpha_p/2}{\log p}}.
\end{multline}

On the other hand, by the inclusion $\AA_{\kappa}(q,\chi) \subset \AA_{\kappa}(q \ddot{Q},\chi)$,
\begin{multline*}
\#\left\{f \in \BB_{\kappa}(q,\chi) : t_f \in [0,T], \ \left|\lambda_f(p)\right| \geq \alpha_p \text{ for all $p \in \PP$}\right\}	\\
\leq \#\left\{f \in \BB_{\kappa}(q \ddot{Q},\chi) : t_f \in [0,T], \ \left|\lambda_f(p)\right| \geq \alpha_p \text{ for all $p \in \PP$}\right\}.
\end{multline*}
Since $q_{\chi \psi^2} \mid q_{\chi}$, we have that $\dot{Q}(q \ddot{Q}, q_{\chi \psi^2}) = 1$. Consequently, \eqref{SarnakGamma0(q)dotq} yields the bound
\begin{multline*}
\#\left\{f \in \BB_{\kappa}(q,\chi) : t_f \in [0,T], \ \left|\lambda_f(p)\right| \geq \alpha_p \text{ for all $p \in \PP$}\right\}	\\
\ll_{\e} \left(\vol(\Gamma_0(q \ddot{Q}) \backslash \Hb) T^2\right)^{1 - 4 \sum_{p \in \PP} \mu_p \frac{\log \alpha_p/2}{\log p} + \e}	\\
\ll_{\e} \left(\vol(\Gamma_0(q) \backslash \Hb) T^2\right)^{1 - 4 \sum_{p \in \PP} \mu_p \frac{\log \alpha_p/2}{\log p} + \e} \ddot{Q}^{1 - 4 \sum_{p \in \PP} \mu_p \frac{\log \alpha_p/2}{\log p}}.
\qedhere
\end{multline*}
\end{proof}

\begin{remark}
Should we wish to improve \eqref{SarnakGamma0} to be uniform in $\PP$, then one needs to take into account the fact that
\begin{multline*}
\prod_{p \in \PP} \left(\frac{\alpha_p}{2}\right)^{-2\ell_p} = \left(\vol(\Gamma_0(q) \backslash \Hb) T^2\right)^{- 4 \sum_{p \in \PP} \mu_p \frac{\log \alpha_p/2}{\log p} + \e} \dot{Q}^{4 \sum_{p \in \PP} \mu_p \frac{\log \alpha_p/2}{\log p}}	\\
\times \prod_{p \in \PP} \left(\frac{\alpha_p}{2}\right)^{2 \left\{\frac{\mu_p \log \left(\vol\left(\Gamma_0(q) \backslash \Hb\right)^2 T^4 \dot{Q}^{-2} \right)}{\log p}\right\}},
\end{multline*}
where $\{x\}$ denotes the fractional part of $x$, and the last term need not necessarily be $\ll_{\e} \left(\vol(\Gamma_0(q) \backslash \Hb) T^2\right)^{\e}$. For this reason, \cite[Proposition 1]{BBR} is not correct in the generality in which it is stated, namely the claim that the result is uniform for $T > p$. Instead, one requires that $p \ll_{\e} T^{\e}$.
\end{remark}

\section{Huxley's Density Theorem for Exceptional Laplacian Eigenvalues}

\hyperref[Huxleythm]{Theorem \ref*{Huxleythm}} is proved similarly to \hyperref[Sarnakthm]{Theorem \ref*{Sarnakthm}}, though we use the Kuznetsov formula \eqref{Kuznetsov} with a carefully chosen test function in place of the pre-Kuznetsov formula \eqref{pre-Kuznetsov}, and we require different methods to bound the Kloosterman term.

\begin{proof}[Proof of {\eqref{HuxleyGamma1}}]
We again use Rankin's trick with nonnegative integers $\ell_p$ and a positive real number $X \geq 1$ to be chosen:
\begin{multline*}
\#\left\{f \in \BB_0\left(\Gamma_1(q)\right) : it_f \in (\alpha_0,1/2), \ \left|\lambda_f(p)\right| \geq \alpha_p \text{ for all $p \in \PP$}\right\}	\\
\leq X^{-2\alpha_0} \prod_{p \in \PP} \alpha_p^{-2\ell_p} \sum_{\substack{f \in \BB_0\left(\Gamma_1(q)\right) \\ it_f \in (0,1/2)}} X^{2it_f} \prod_{p \in \PP} \left|\lambda_f(p)\right|^{2\ell_p}.
\end{multline*}
Again using \eqref{basisGamma1(q)} and \eqref{rho1lowerbound},
\begin{multline*}
\sum_{\substack{f \in \BB_0\left(\Gamma_1(q)\right) \\ it_f \in (0,1/2)}} X^{2it_f} \prod_{p \in \PP} \left|\lambda_f(p)\right|^{2\ell_p}	\\
= \sum_{\substack{\chi \hspace{-.25cm} \pmod{q} \\ \chi(-1) = 1}} \sum_{\substack{q_1 q_2 = q \\ q_1 \equiv 0 \hspace{-.25cm} \pmod{q_{\chi}}}} \sum_{\substack{f \in \BB_0^{\ast}\left(q_1,\chi\right) \\ it_f \in (0,1/2)}} \tau(q_2) X^{2it_f} \prod_{p \in \PP} \left|\lambda_f(p)\right|^{2\ell_p}	\\
\ll_{\e} q^{1 + \e} \sum_{\substack{\chi \hspace{-.25cm} \pmod{q} \\ \chi(-1) = 1}} \sum_{\substack{q_1 q_2 = q \\ q_1 \equiv 0 \hspace{-.25cm} \pmod{q_{\chi}}}} \sum_{\substack{f \in \BB_0^{\ast}\left(q_1,\chi\right) \\ it_f \in (0,1/2)}} \xi_f \nu_f X^{2it_f} \prod_{p \in \PP} \left|\lambda_f(p)\right|^{2\ell_p}	\\
\end{multline*}

We take $m = 1$, $n = \prod_{p \in \PP} p^{2j_p}$, and
\[h(t) = h_X(t) = \left(\frac{X^{it} + X^{-it}}{t^2 + 1}\right)^2\]
in the Kuznetsov formula \eqref{Kuznetsov}, multiply both sides by $\prod_{p \in \PP} \alpha_{2j_p,2\ell_p} \overline{\chi}(p)^{j_p}$, and sum over all $0 \leq j_p \leq \ell_p$, over all $p \in \PP$, and over all even Dirichlet characters modulo $q$. On the spectral side, we discard all but the discrete spectrum for which $it_f \in (0,1/2)$ via positivity, so that the spectral side is bounded from below by a constant multiple of
\[\sum_{\substack{\chi \hspace{-.25cm} \pmod{q} \\ \chi(-1) = 1}} \sum_{\substack{q_1 q_2 = q \\ q_1 \equiv 0 \hspace{-.25cm} \pmod{q_{\chi}}}} \sum_{\substack{f \in \BB_0^{\ast}\left(q_1,\chi\right) \\ it_f \in (0,1/2)}} \xi_f \nu_f X^{2it_f} \prod_{p \in \PP} \left|\lambda_f(p)\right|^{2\ell_p}.\]
We only pick up the delta term on the geometric side when $j_p = 0$ for all $p \in \PP$, in which case the term is bounded by a constant multiple of $q \prod_{p \in \PP} 2^{2\ell_p}$. We write the Kloosterman term in the form
\begin{multline*}
\frac{\varphi(q)}{2\pi i} \sum_{\substack{j_p = 0 \\ p \in \PP}}^{\ell_p} \prod_{p \in \PP} \alpha_{2j_p,2\ell_p} \int_{\sigma - i\infty}^{\sigma + i\infty} \sum_{\substack{c = 1 \\ c \equiv 0 \hspace{-.25cm} \pmod{q}}}^{\infty} \frac{\Re\left(S_{\prod_{p \in \PP} p^{j_p}(q)}\left(1,\prod_{p \in \PP} p^{2j_p};c\right)\right)}{c}	\\
\times J_s\left(\frac{4\pi \prod_{p \in \PP} p^{j_p}}{c}\right) \frac{s h_X\left(\frac{is}{2}\right)}{\cos \frac{\pi s}{2}} \, ds
\end{multline*}
for any $1/2 < \sigma < 1$. We have, via \cite[8.411.4]{GR}, the bound
\[J_s(x) \ll \frac{x^{\sigma}}{\left|\Gamma\left(s + \frac{1}{2}\right)\right|} \ll e^{\pi|s|/2} \left(\frac{x}{|s|}\right)^{\sigma},\]
and so the integral in the Kloosterman term is bounded by a constant multiple of
\[\prod_{p \in \PP} p^{j_p \sigma} \sum_{\substack{c = 1 \\ c \equiv 0 \hspace{-.25cm} \pmod{q}}}^{\infty} \frac{\left|S_{\prod_{p \in \PP} p^{j_p} (q)}\left(1,\prod_{p \in \PP} p^{2j_p};c\right)\right|}{c^{1 + \sigma}} \int_{\sigma/2 - i\infty}^{\sigma/2 + i\infty} \left|r^{3/4} h_X(ir)\right| \, dr.\]
We take
\[\sigma = \frac{1}{2} + \frac{1}{\log \left(X \prod_{p \in \PP} p^{\ell_p}\right)},\]
so that the integral is bounded by a constant multiple of $\sqrt{X}$, and use \eqref{Gamma1(q)csigma} to bound the summation over $c$ and \eqref{sumalphabound} to bound the summation over $j_p$ and $p \in \PP$ in order to find that
\begin{align*}
& \#\left\{f \in \BB_0\left(\Gamma_1(q)\right) : it_f \in (\alpha_0,1/2), \ \left|\lambda_f(p)\right| \geq \alpha_p \text{ for all $p \in \PP$}\right\}	\\
& \qquad\ll_{\e} q^{1 + \e} X^{-2\alpha_0} \prod_{p \in \PP} \left(\frac{\alpha_p}{2}\right)^{-2\ell_p}	\\
& \hspace{2.5cm} \times \left(q + \sqrt{X} \prod_{p \in \PP} p^{\ell_p/2} \left(\log \left(X \prod_{p \in \PP} p^{\ell_p}\right)\right)^2 \frac{1}{\sqrt{q}} \prod_{p' \mid q} \frac{1}{1 - {p'}^{-1/2}}\right).
\end{align*}
The result follows upon taking
\[X = \vol\left(\Gamma_1(q) \backslash \Hb\right)^{3\mu_0/2}, \qquad \ell_p = \left\lfloor \frac{\mu_p \log \vol\left(\Gamma_1(q) \backslash \Hb\right)^{3/2}}{\log p} \right\rfloor.\qedhere\]
\end{proof}

\begin{proof}[Proof of {\eqref{HuxleyGamma}}]
By using \eqref{basisGamma(q)} and \eqref{Gamma(q)csigma} in place of \eqref{basisGamma1(q)} and \eqref{Gamma1(q)csigma}, we obtain
\begin{align*}
& \#\left\{f \in \BB_0\left(\Gamma(q)\right) : it_f \in (\alpha_0,1/2), \ \left|\lambda_f(p)\right| \geq \alpha_p \text{ for all $p \in \PP$}\right\}	\\
& \qquad\ll_{\e} q^{2 + \e} X^{-2\alpha_0} \prod_{p \in \PP} \left(\frac{\alpha_p}{2}\right)^{-2\ell_p}	\\
& \hspace{2.5cm} \times \left(q + \sqrt{X} \prod_{p \in \PP} p^{\ell_p/2} \left(\log \left(X \prod_{p \in \PP} p^{\ell_p}\right)\right)^2 \frac{1}{q} \prod_{p' \mid q} \frac{1}{1 - {p'}^{-1/2}}\right),
\end{align*}
and it remains to take
\[X = \vol\left(\Gamma(q) \backslash \Hb\right)^{4\mu_0/3}, \qquad \ell_p = \left\lfloor \frac{\mu_p \log \vol\left(\Gamma(q) \backslash \Hb\right)^{4/3}}{\log p} \right\rfloor.\qedhere\]
\end{proof}

\begin{proof}[Proof of {\eqref{HuxleyGamma0}}]
We use \eqref{basisqchi} and \eqref{Gamma0(q)csigma} in place of \eqref{basisGamma1(q)} and \eqref{Gamma1(q)csigma}, so that
\begin{multline*}
\#\left\{f \in \BB_0(q,\chi) : it_f \in (\alpha_0,1/2), \ \left|\lambda_f(p)\right| \geq \alpha_p \text{ for all $p \in \PP$}\right\}	\\
\ll_{\e} q^{1 + \e} X^{-2\alpha_0} \prod_{p \in \PP} \left(\frac{\alpha_p}{2}\right)^{-2\ell_p} \left(1 + \sqrt{X} \prod_{p \in \PP} p^{\ell_p/2} \left(\log \left(X \prod_{p \in \PP} p^{\ell_p}\right)\right)^2 \frac{2^{\omega(q)} \dot{Q}}{\varphi(q)}\right).
\end{multline*}
We find that
\begin{multline*}
\#\left\{f \in \BB_0(q,\chi) : it_f \in (\alpha_0,1/2), \ \left|\lambda_f(p)\right| \geq \alpha_p \text{ for all $p \in \PP$}\right\}	\\
\ll_{\e} \vol(\Gamma_0(q) \backslash \Hb)^{1 - 4 \left(\mu_0 \alpha_0 + \sum_{p \in \PP} \mu_p \frac{\log \alpha_p/2}{\log p}\right) + \e} \dot{Q}^{4 \left(\mu_0 \alpha_0 + \sum_{p \in \PP} \mu_p \frac{\log \alpha_p/2}{\log p}\right)}.
\end{multline*}
 by taking
\[X = \vol\left(\Gamma_0(q) \backslash \Hb\right)^{2\mu_0} \dot{Q}^{-2\mu_0}, \qquad \ell_p = \left\lfloor \frac{\mu_p \log \left(\vol\left(\Gamma_0(q) \backslash \Hb\right)^2 \dot{Q}^{-2}\right)}{\log p} \right\rfloor.\]
Again, we also have that
\begin{multline*}
\#\left\{f \in \BB_0(q,\chi) : it_f \in (\alpha_0,1/2), \ \left|\lambda_f(p)\right| \geq \alpha_p \text{ for all $p \in \PP$}\right\}	\\
\leq \#\left\{f \in \BB_0(q \ddot{Q},\chi \psi^2) : it_f \in (\alpha_0,1/2), \ \left|\lambda_f(p)\right| \geq \alpha_p \text{ for all $p \in \PP$}\right\}
\end{multline*}
for any primitive character $\psi$ modulo $\ddot{Q}$, which implies that
\begin{multline*}
\#\left\{f \in \BB_0(q,\chi) : it_f \in (\alpha_0,1/2) \in [0,T], \ \left|\lambda_f(p)\right| \geq \alpha_p \text{ for all $p \in \PP$}\right\}	\\
\ll_{\e} \vol(\Gamma_0(q) \backslash \Hb)^{1 - 4 \left(\mu_0 \alpha_0 + \sum_{p \in \PP} \mu_p \frac{\log \alpha_p/2}{\log p}\right) + \e} \ddot{Q}^{1 - 4 \left(\mu_0 \alpha_0 + \sum_{p \in \PP} \mu_p \frac{\log \alpha_p/2}{\log p}\right)}.
\qedhere
\end{multline*}
\end{proof}

\section{Improving \texorpdfstring{\hyperref[Sarnakthm]{Theorems \ref*{Sarnakthm}}}{Theorems \ref{Sarnakthm}} and \ref{Huxleythm} for \texorpdfstring{$\Gamma_1(q)$}{\83\223 \9040\201(q)} via Twisting}

In this section, we prove \hyperref[improvedthm]{Theorem \ref*{improvedthm}}. Let $f \in \BB_{\kappa}^{\ast}(q,\chi)$ be a newform, and for a primitive character $\psi$ modulo $q_{\psi}$ with $q_{\psi} \mid q$, we let $f \otimes \psi$ denote the twist of $f$ by $\psi$; this is the \emph{newform} whose Hecke eigenvalues $\lambda_{f \otimes \psi}(n)$ are equal to $\lambda_f(n) \psi(n)$ whenever $(n,q) = 1$. By \cite[Proposition 3.1]{AL}, the weight of $f \otimes \psi$ is $\kappa$, the level of $f \otimes \psi$ divides $q^2$, and the nebentypus is the primitive character that induces $\chi \psi^2$. We make crucial use of the fact that twisting by a Dirichlet character preserves the Laplacian eigenvalue $\lambda_f = 1/4 + t_f^2$ and the absolute value $|\lambda_f(n)|$ of the Hecke eigenvalues of $f$ for all $(n,q) = 1$. Moreover, if $f_1 \in \BB_{\kappa}^{\ast}\left(q_1,\chi_1\right)$, $f_2 \in \BB_{\kappa}^{\ast}\left(q_2,\chi_2\right)$ are such that there exist primitive Dirichlet characters $\psi_1$ modulo $q_{\psi_1}$ and $\psi_2$ modulo $q_{\psi_2}$ with $q_{\psi_1}, q_{\psi_2} \mid q$ such that
\[f_1 \otimes \psi_1 = f_2 \otimes \psi_2,\]
then $f_2 = f_1 \otimes \psi_1 \overline{\psi_2}$.

\begin{lemma}\label{twistinglevellemma}
If $q$ is squarefree, $\psi$ is a primitive Dirichlet modulo $q_{\psi}$, where $q_{\psi} \mid q$, and $f \in \BB_{\kappa}^{\ast}(q,\chi)$, then the level of $f \otimes \psi$ divides $q$ if and only if $\overline{\psi}$ divides $\chi$, in the sense that $\psi \chi$ has conductor dividing $q_{\chi}$.
\end{lemma}

\begin{proof}
This follows via the methods of \cite{Hum}. For $p \mid q$, let $\pi_p$ be the local component of the cuspidal automorphic representation $\pi$ of $\GL_2(\A_{\Q})$ associated to the newform $f$, so that the central character $\omega_p$ of $\pi_p$ is the local component of the Hecke character $\omega$ that is the id\`{e}lic lift of $\chi$. As $q$ is squarefree, $\pi_p$ is either a principal series representation or a special representation.

In the former case, $\pi_p = \omega_{p,1} \boxplus \omega_{p,2}$ with central character $\omega_p = \omega_{p,1} \omega_{p,2}$, where $\omega_{p,1},\omega_{p,2}$ are characters of $\Q_p^{\times}$ with conductor exponents $c(\omega_{p,1}), c(\omega_{p,2}) \in \{0,1\}$ such that the conductor exponent $c(\pi_p)$ of $\pi_p$ is $c(\omega_{p,1}) + c(\omega_{p,2}) = 1$. The twist $\pi_p \otimes \omega_p'$ of $\pi_p$ by a character $\omega_p'$ of $\Q_p^{\times}$ of conductor exponent $c(\omega_p') \in \{0,1\}$ is $\omega_{p,1} \omega_p' \boxplus \omega_{p,2} \omega_p'$ with corresponding conductor exponent $c(\pi_p \otimes \omega_p') = c(\omega_{p,1} \omega_p') + c(\omega_{p,2} \omega_p')$. For this to be at most $1$, either $\omega_p'$ is unramified, or one of $c(\omega_{p,1} \omega_p'), c(\omega_{p,2} \omega_p')$ must be equal to $0$, so that $\overline{\omega_p'}$ is equal to $\omega_{p,1}$ or $\omega_{p,2}$ up to multiplication by an unramified character.

In the latter case, $\pi_p = \omega_{p,1} \St$ with central character $\omega_p = \omega_{p,1}^2$ such that $c(\omega_{p,1}) = 0$, so that $c(\pi_p) = 1$. The twist of $\pi_p$ by $\omega_p'$ is $\omega_{p,1} \omega_p' \St$, with corresponding conductor exponent $c(\pi_p \otimes \omega_p') = \max\{1, 2c(\omega_{p,1} \omega_p')\}$. For this to be at most $1$, $\omega_p'$ must be unramified.

It follows that if the Hecke character $\omega'$ is the id\`{e}lic lift of $\psi$, then the conductor of $\pi \otimes \omega'$ divides $q$ if and only if the conductor of $\omega' \omega$ divides the conductor of $\omega$.
\end{proof}

From this, we have the following.

\begin{corollary}\label{twistequalcorollary}
Let $q$ be squarefree. Given a newform $g$ of level dividing $q^2$, there exist at most $\tau(q)$ newforms $f$ of level dividing $q$ that can be twisted by a Dirichlet character of conductor dividing $q$ to give $g$.
\end{corollary}

\begin{proof}
Suppose that $f_1 \in \BB_{\kappa}^{\ast}(q_1,\chi_1)$ and $f_2 \in \BB_{\kappa}^{\ast}(q_2,\chi_2)$ with $q_1$ and $q_2$ dividing $q$ are such that there exist Dirichlet characters $\psi_1$ and $\psi_2$ of conductors dividing $q$ for which $f_1 \otimes \psi_1 = f_2 \otimes \psi_2 = g$. Then $f_2 = f_1 \otimes \psi_1 \overline{\psi_2}$, and \hyperref[twistinglevellemma]{Lemma \ref*{twistinglevellemma}} implies that $\overline{\psi_1} \psi_2$ divides $\chi_1$. Since the conductor of $\chi_1$ divides $q_1$, the level of $f_1$, the proof is complete by noting that the number of Dirichlet characters $\psi_2$ modulo $q$ for which this may occur is bounded by the number of divisors of $q$.
\end{proof}

\begin{lemma}
Let $q$ be squarefree, let $\PP$ be a finite collection of primes not dividing $q$, let $E_0$ be a measurable subset of $[0,\infty) \cup i(0,1/2)$, and let $E_p$ be a measurable subset of $[0,\infty)$ for each $p \in \PP$. Then
\begin{multline*}
\#\left\{f \in \BB_{\kappa}\left(\Gamma_1(q)\right) : t_f \in E_0, \ |\lambda_f(p)| \in E_p \text{ for all $p \in \PP$}\right\}	\\
\leq \frac{\tau(q)^2}{\varphi(q)} 
\#\left\{f \in \BB_{\kappa}\left(\Gamma(q)\right) : t_f \in E_0, \ |\lambda_f(p)| \in E_p \text{ for all $p \in \PP$}\right\}.
\end{multline*}
\end{lemma}

\begin{proof}
From \eqref{basisGamma1(q)},
\[\#\left\{f \in \BB_{\kappa}\left(\Gamma_1(q)\right) : t_f \in E_0, \ |\lambda_f(p)| \in E_p \text{ for all $p \in \PP$}\right\}\]
is equal to
\[\sum_{\substack{\chi \hspace{-.25cm} \pmod{q} \\ \chi(-1) = (-1)^{\kappa}}} \sum_{\substack{q_1 q_2 = q \\ q_1 \equiv 0 \hspace{-.25cm} \pmod{q_{\chi}}}} \tau(q_2) \#\left\{f \in \BB_{\kappa}^{\ast}(q_1,\chi) : t_f \in E_0, \ |\lambda_f(p)| \in E_p \text{ for all $p \in \PP$}\right\},\]
which, in turn, is equal to
\begin{multline*}
\frac{1}{\varphi(q)} \sum_{\psi \hspace{-.25cm} \pmod{q}} \sum_{\substack{\chi \hspace{-.25cm} \pmod{q} \\ \chi(-1) = (-1)^{\kappa}}} \sum_{\substack{q_1 q_2 = q \\ q_1 \equiv 0 \hspace{-.25cm} \pmod{q_{\chi}}}} \tau(q_2)	\\
\times \#\left\{f \otimes \psi : f \in \BB_{\kappa}^{\ast}(q_1,\chi), \ t_f \in E_0, \ |\lambda_f(p)| \in E_p \text{ for all $p \in \PP$}\right\},
\end{multline*}
as twisting preserves Laplacian eigenvalues and the absolute value of Hecke eigenvalues. Each twist $g = f \otimes \psi$ of some $f \in \BB_{\kappa}^{\ast}(q_1,\chi)$ is a newform of weight $\kappa$, level dividing $q^2$, and nebentypus of conductor dividing $q$, and \hyperref[twistequalcorollary]{Corollary \ref*{twistequalcorollary}} implies that there are at most $\tau(q)$ newforms of level dividing $q$ that can be twisted by a Dirichlet character of conductor dividing $q$ to yield $g$. Since $\tau(q_2) \leq \tau(q)$, the above quantity is bounded by
\[\frac{\tau(q)^2}{\varphi(q)} \sum_{\substack{\chi \hspace{-.25cm} \pmod{q} \\ \chi(-1) = (-1)^{\kappa}}} \sum_{\substack{q_1 q_2 = q^2 \\ q_1 \equiv 0 \hspace{-.25cm} \pmod{q_{\chi}}}} \#\left\{g \in \BB_{\kappa}^{\ast}(q_1,\chi) : t_g \in E_0, \ |\lambda_g(p)| \in E_p \text{ for all $p \in \PP$}\right\},\]
while the explicit basis \eqref{basisGamma(q)} of $\BB_{\kappa}(\Gamma(q))$ implies that
\[\#\left\{g \in \BB_{\kappa}(\Gamma(q)) : t_g \in E_0, \ |\lambda_g(p)| \in E_p \text{ for all $p \in \PP$}\right\}\]
is equal to
\[\sum_{\substack{\chi \hspace{-.25cm} \pmod{q} \\ \chi(-1) = (-1)^{\kappa}}} \sum_{\substack{q_1 q_2 = q^2 \\ q_1 \equiv 0 \hspace{-.25cm} \pmod{q_{\chi}}}} \tau(q_2) \# \left\{g \in \BB_{\kappa}^{\ast}(q_1,\chi) : t_g \in E_0, \ |\lambda_g(p)| \in E_p \text{ for all $p \in \PP$}\right\}.\]
This yields the result.
\end{proof}

Combining this with the fact that $\vol\left(\Gamma(q) \backslash \Hb\right) = q \vol\left(\Gamma_1(q) \backslash \Hb\right)$, we deduce \hyperref[improvedthm]{Theorem \ref*{improvedthm}}. It is likely that a more careful analysis could obtain this same result even when $q$ is not squarefree via the methods in \cite{Hum}.

\section*{}
\vspace{-.6cm}

\subsection*{Acknowledgements}

The author thanks Peter Sarnak for many helpful discussions on this topic, as well as the referee for correcting several mistakes in an earlier version of this paper.


\begin{thebibliography}{BrMia09}

\bibitem[AL78]{AL} A.\ O.\ L.\ Atkin and Wen-Ch'ing Winnie Li, ``Twists of Newforms and Pseudo-Eigenvalues of $W$-Operators'', \textit{Inventiones Mathematicae} \textbf{48}:3 (1978), 221--243. \textsc{doi}:\allowbreak\href{https://doi.org/10.1007/BF01390245}{10.1007/BF01390245}

\bibitem[BBR14]{BBR} Valentin Blomer, Jack Buttcane, and Nicole Raulf, ``A Sato--Tate Law for $\GL(3)$'', \textit{Commentarii Mathematici Helvetici} \textbf{89}:4 (2014), 895--919. \textsc{doi}:\allowbreak\href{https://doi.org/10.4171/CMH/337}{10.4171/CMH/337}

\bibitem[BHM07]{BHM} Valentin Blomer, Gergely Harcos, and Philippe Michel, ``Bounds for Modular $L$-Functions in the Level Aspect'', \textit{Annales Scientifiques de l'\'{E}cole Normale Sup\'{e}rieure, $4^{\mathrm{e}}$ s\'{e}rie} \textbf{40}:5 (2007), 697--740. \textsc{doi}:\allowbreak\href{https://doi.org/10.1016/j.ansens.2007.05.003}{10.1016/j.ansens.2007.05.003}

\bibitem[BlMil15]{BlMil} Valentin Blomer and Djordje Mili\'{c}evi\'{c}, ``The Second Moment of Twisted Modular $L$-Functions'', \textit{Geometric and Functional Analysis} \textbf{25}:2 (2015), 453--516. \textsc{doi}:\allowbreak\href{https://doi.org/10.1007/s00039-015-0318-7}{10.1007/s00039-015-0318-7}

\bibitem[BS07]{BS} Andrew R.\ Booker and Andreas Str\"{o}mbergsson, ``Numerical Computations with the Trace Formula and the Selberg Eigenvalue Conjecture'', \textit{Journal f\"{u}r die reine und angewandte Mathematik} \textbf{607} (2007), 113--161. \textsc{doi}:\allowbreak\href{https://doi.org/10.1515/CRELLE.2007.047}{10.1515/CRELLE.2007.047}

\bibitem[BrMia09]{BrMia} Roelof W. Bruggeman and Roberto J. Miatello, \textit{Sum Formula for $\SL_2$ over a Totally Real Number Field}, Memoirs of the American Mathematical Society \textbf{197}:919, American Mathematical Society, Providence, 2009. \textsc{doi}:\allowbreak\href{https://doi.org/10.1090/memo/0919}{10.1090/memo/0919}

\bibitem[CDF97]{CDF} J.\ B.\ Conrey, W.\ Duke, and D.\ W.\ Farmer, ``The Distribution of the Eigenvalues of Hecke Operators'', \textit{Acta Arithmetica} \textbf{78}:4 (1997), 405--409. \textsc{doi}:\allowbreak\href{https://doi.org/10.4064/aa-78-4-405-409}{10.4064/aa-78-4-405-409}

\bibitem[DFI02]{DFI} W.\ Duke, J.\ B.\ Friedlander, and H.\ Iwaniec, ``The Subconvexity Problem for Artin $L$-Functions'', \textit{Inventiones Mathematicae} \textbf{149}:3 (2002), 489--577. \textsc{doi}:\allowbreak\href{https://doi.org/10.1007/s002220200223}{10.1007/s002220200223}

\bibitem[GR07]{GR} I.\ S.\ Gradshteyn and I.\ M.\ Ryzhik, \textit{Table of Integrals, Series, and Products, Seventh Edition}, editors Alan Jeffrey and Daniel Zwillinger, Academic Press, Burlington, 2007.

\bibitem[HM07]{HM} C.\ P.\ Hughes and Steven J.\ Miller, ``Low-Lying Zeros of $L$-Functions with Orthogonal Symmetry'', \textit{Duke Mathematical Journal} \textbf{136}:1 (2007), 115--172. \textsc{doi}:\allowbreak\href{https://doi.org/10.1215/S0012-7094-07-13614-7}{10.1215/S0012-7094-07-13614-7}

\bibitem[Hum17]{Hum} Peter Humphries, ``Spectral Multiplicity for Maa\ss{} Newforms of Non-Squarefree Level'', to appear in \textit{International Mathematics Research Notices} (2017), 41 pages. \textsc{doi}:\allowbreak\href{https://doi.org/10.1093/imrn/rnx283}{10.1093/imrn/rnx283}

\bibitem[Hux86]{Hux} M.\ N.\ Huxley, ``Exceptional Eigenvalues and Congruence Subgroups'', in \textit{The Selberg Trace Formula and Related Topics}, editors Dennis A.\ Hejhal, Peter Sarnak and Audrey Anne Terras, Contemporary Mathematics \textbf{53}, American Mathematical Society, Providence, 1986, 341--349. \textsc{doi}:\allowbreak\href{https://doi.org/10.1090/conm/053/853564}{10.1090/conm/053/853564}

\bibitem[Iwa02]{Iwa} Henryk Iwaniec, \textit{Spectral Methods of Automorphic Forms, Second Edition}, Graduate Studies in Mathematics \textbf{53}, American Mathematical Society, Providence, 2002. \textsc{doi}:\allowbreak\href{https://doi.org/10.1090/gsm/053}{10.1090/gsm/053}

\bibitem[IK04]{IK} Henryk Iwaniec and Emmanuel Kowalski, \textit{Analytic Number Theory}, American Mathematical Society Colloquium Publications \textbf{53}, American Mathematical Society, Providence, 2004. \textsc{doi}:\allowbreak\href{https://doi.org/10.1090/coll/053}{10.1090/coll/053}

\bibitem[ILS00]{ILS} Henryk Iwaniec, Wenzhi Luo, and Peter Sarnak, ``Low Lying Zeros of Families of $L$-Functions'', \textit{Publications Math\'{e}matiques de l'Institut des Hautes \'{E}tudes Scientifiques} \textbf{91}:1 (2000), 55--131. \textsc{doi}:\allowbreak\href{https://doi.org/10.1007/BF02698741}{10.1007/BF02698741}

\bibitem[Kim03]{Kim} Henry H.\ Kim, ``Functoriality for the Exterior Square of $\GL_4$ and the Symmetric Fourth of $\GL_2$'', Appendix 1 by Dinakar Ramakrishnan and Appendix 2 by Henry H.\ Kim and Peter Sarnak, \textit{Journal of the American Mathematical Society} \textbf{16}:1 (2003), 139--183. \textsc{doi}:\allowbreak\href{https://doi.org/10.1090/S0894-0347-02-00410-1}{10.1090/S0894-0347-02-00410-1}

\bibitem[KL13]{KL} A.\ Knightly and C.\ Li, \textit{Kuznetsov's Trace Formula and the Hecke Eigenvalues of Maass Forms}, Memoirs of the American Mathematical Society \textbf{224}:1055, American Mathematical Society, Providence, 2013. \textsc{doi}:\allowbreak\href{https://doi.org/10.1090/S0065-9266-2012-00673-3}{10.1090/S0065-9266-2012-00673-3}

\bibitem[Kuz81]{Kuz} N.\ V.\ Kuznetsov, ``The Petersson Conjecture for Cusp Forms of Weight Zero and the Linnik Conjecture. Sums of Kloosterman Sums'', \textit{Mathematics of the USSR. Sbornik} \textbf{39} (1981), 299--342. \textsc{doi}:\allowbreak\href{https://doi.org/10.1070/SM1981v039n03ABEH001518}{10.1070/SM1981v039n03ABEH001518}

\bibitem[Li09]{Li} Xiannan Li, ``Upper Bounds on $L$-Functions at the Edge of the Critical Strip'', \textit{International Mathematics Research Notices} \textbf{2010}:4 (2010), 727--755. \textsc{doi}:\allowbreak\href{https://doi.org/10.1093/imrn/rnp148}{10.1093/imrn/rnp148}

\bibitem[LM13]{LMFDB} The LMFDB Collaboration, \textit{The $L$-Functions and Modular Forms Database}, \url{http://www.lmfdb.org}, 2013, [Online; accessed 30 January 2017].

\bibitem[Mag13]{Mag} P\'{e}ter Maga, ``A Semi-Adelic Kuznetsov Formula over Number Fields'', \textit{International Journal of Number Theory} \textbf{9}:7 (2013), 1649--1681. \textsc{doi}:\allowbreak\href{https://doi.org/10.1142/S1793042113500498}{10.1142/S1793042113500498}

\bibitem[PY18]{PY} Ian Petrow and Matthew P.\ Young, ``A Generalized Cubic Moment and the Petersson Formula for Newforms'', to appear in \textit{Mathematische Annalen} (2018), 67 pages. \textsc{doi}:\allowbreak\href{https://doi.org/10.1007/s00208-018-1745-1}{10.1007/s00208-018-1745-1}

\bibitem[Sar87]{Sarnak} Peter Sarnak, ``Statistical Properties of Eigenvalues of the Hecke Operators'', in \textit{Analytic Number Theory and Diophantine Problems. Proceedings of a Conference at Oklahoma State University, 1984}, editors A.\ C.\ Adolphson, J.\ B.\ Conrey, A.\ Ghosh, and R.\ I.\ Yager, Progress in Mathematics \textbf{70}, Birkh\"{a}user, Boston, 1987, 321--331. \textsc{doi}:\allowbreak\href{https://doi.org/10.1007/978-1-4612-4816-3_19}{10.1007/978-1-4612-4816-3\_19}

\bibitem[S-PY18]{S-PY} Rainer Schulze-Pillot and Abdullah Yenirce, ``Petersson Products of Bases of Spaces of Cusp Forms and Estimates for Fourier Coefficients'', \textit{International Journal of Number Theory} \textbf{14}:8 (2018), 2277--2290. \textsc{doi}:\allowbreak\href{https://doi.org/10.1142/S1793042118501385}{10.1142/S1793042118501385}

\end{thebibliography}
\end{document}